\documentclass[a4paper,reqno,11pt]{amsart}
\usepackage{latexsym, amssymb, enumerate, amsmath}

\usepackage[pdftex]{graphicx}
\usepackage{lscape}

\usepackage{amsfonts}
\usepackage{amssymb}
\usepackage{fancyhdr}
\usepackage{setspace}
\usepackage{color}
\usepackage{enumitem}

\newcommand{\ds}{\displaystyle }

\newtheorem{thm}{Theorem}[section]
\newtheorem{prop}[thm]{Proposition}

\newtheorem{defn}[thm]{Definition}

\newtheorem{rem}[thm]{Remark}

\begin{document}

\title[A well-balanced numerical scheme for a model of chemotaxis]{A well-balanced numerical scheme for a  one dimensional quasilinear hyperbolic model of chemotaxis}
\author[R. Natalini, M. Ribot, and M. Twarogowska]{R. Natalini$^1$ \and M. Ribot$^2$ \and M. Twarogowska$^3$}

\thanks{ {\emph{keywords and phrases: }}hyperbolic system with source, chemotaxis, stationary solutions with vacuum, finite volume methods,  well-balanced scheme}

\thanks{\vspace{0.2cm}          
$^1$ Istituto per le Applicazioni del Calcolo ``Mauro Picone'', 
Consiglio Nazionale delle Ricerche,  via dei Taurini 19, I-00185 Roma, Italy
({\tt roberto.natalini@cnr.it})}.  

\thanks{$^2$ Laboratoire J. A. Dieudonn\'e, UMR CNRS 6621, Universit\'e de Nice-Sophia Antipolis,
 Parc Valrose, F-06108 Nice Cedex 02, France
\&  Project Team COFFEE, INRIA Sophia Antipolis, France
 ({\tt ribot@unice.fr}).}     
 
\thanks{$^3$ INRIA Sophia Antipolis - M\'editerran\'ee, 
OPALE Project-Team,
2004, route des Lucioles -- BP 93, 
06902 Sophia Antipolis Cedex, France ({\tt monika.twarogowska@inria.fr}).}

\pagestyle{myheadings} \markboth{Well-balanced scheme for a 1D quasilinear hyperbolic model of chemotaxis}{R. Natalini, M. Ribot, M. Twarogowska }\maketitle

\begin{abstract}
We introduce a numerical scheme to approximate a quasi-linear hyperbolic system  which models the movement of cells under the influence of chemotaxis. Since we expect to find  solutions which contain vacuum parts, we propose an upwinding scheme which handles properly the presence of vacuum and, besides,  which gives a good approximation of the time asymptotic states of the system. For this scheme we prove some basic analytical properties and study its stability near some of the steady states of the system. Finally, we present some numerical simulations which show the dependence of the asymptotic behavior of the solutions upon  the parameters of the system.\newline\newline
{\bf AMS Primary: 65M08; Secondary:  35L60,  92B05, 92C17.}
\end{abstract}

\section{Introduction}\label{sec:introduction}

The movement of bacteria, cells or  other microorganisms under the effect of a chemical stimulus, represented by a chemoattractant, has been widely studied in mathematics in the last two decades, see \cite{H, M1,M2, Pe}, and numerous models involving partial differential equations have been proposed. 
The basic unknowns in these chemotactic models are the density of individuals and the concentrations of some chemical attractants. One of the most considered models is the Patlak-Keller-Segel system \cite{KS}, where the evolution of the density of cells is described by a parabolic equation, and the concentration of  a chemoattractant is generally given by a parabolic or elliptic equation, depending on the different regimes to be described and on authors' choices. The behavior of this system is quite well known now: in the one-dimensional case, the solution is always global in time \cite{NY07}, while in two and more dimensions the solutions exist globally in time
or blow up according to the size of the initial data, see \cite{CC08,CCE12} and references therein. However, a drawback of this model is that the diffusion leads alternatively to a fast dissipation or an explosive behavior,  and prevents us to observe intermediate organized structures, like aggregation patterns.

In this paper, we consider   a  quasi-linear hyperbolic system of chemotaxis  introduced by Gamba et al. \cite{Gamba2} to describe the early stages of the vasculogenesis process,  namely  the formation of blood vessels networks during the embryonic development. 
The model forms a  hyperbolic--parabolic system for the following unknowns: the density of endothelial cells $\rho(x,t)$,  their momentum  $\rho u(x,t)$ and  the concentration  $\phi(x,t)$ of a chemoattractant. In one space dimension the system reads
\begin{equation}\label{eq:main_system}
\left\{\begin{array}{l}
\displaystyle{\rho_{t}+(\rho u)_{x}=0,}\\
\displaystyle{(\rho u)_{t}+\left(\rho u^{2}+P(\rho)\right)_{x}=-\alpha\rho u+\chi\rho\phi_{x},}\\
\displaystyle{\phi_{t}=D\phi_{xx}+a\rho-b\phi}.
\end{array}\right.
\end{equation}
 Loosely speaking, the movement  of cells is directed by the gradient of the chemical mediator and is slowed down by the adhesion with the substratum. The positive constants $\chi$  and $\alpha$ measure respectively  the strength of the cells response to the concentration of the chemical substance and   the strength of the friction forces. Overcrowding of cells is prevented by a phenomenological density dependent pressure function given by the pressure law for isentropic gases
\begin{equation}\label{eq:pressure_law}
P(\rho)=\varepsilon\rho^{\gamma},\qquad\gamma>1,\varepsilon>0.
\end{equation}
Besides, the evolution of chemoattractant is given by a linear diffusion equation with a source term which depends on $\rho$: the chemoattractant is released by the cells, diffuses in the environment and is linearly degraded. The positive parameters $D,a,b$ are respectively its diffusion coefficient, the production rate
and  the degradation rate and the production term is proportional to the cell density.

This model was introduced to mimic the results of  in vitro experiments  performed by Serini et al. \cite{Gamba1} using human vascular, endothelial cells. Randomly seeded on the plain gel substratum, these cells migrate and interact together via chemotaxis signaling and, after a while, they aggregate to 
  form a network of capillaries. To be more precise,   different final stages are observed depending on the size of the  initial density of cells. For low densities,  only isolated, disconnected clusters are formed. 
Increasing the number of cells,  a sharp  percolative transition occurs and a capillary-like network appears with a characteristic chord length independent of the size of the initial density. Further increase of the number of cells leads to a continuum crossover characterized by the accumulation of additional cells on the network chords until the structure is no longer visible. Finally, for very high initial densities, 
a continuous carpet of cells with holes, the so called ''Swiss cheese'' configuration, is observed.
From a mathematical point of view, 
emerging structured patterns, such as capillary-like networks, 
may be seen as the appearance of nonconstant asymptotic 
solutions   with vacuum, namely solutions composed of regions where the density is strictly positive and of  regions where the density (of cells) vanishes.

To reproduce the biological setting, we  consider system (\ref{eq:main_system}) on a bounded domain $[0,L]$  with no-flux boundary conditions, namely, for all time $t>0$,  $\ds \rho_{x}(0,t)=\rho_{x}(L,t)=0$ for the density, $ \rho u(0,t)= \rho u(L,t)=0$ for the momentum and $\phi_{x}(0,t)=\phi_{x}(L,t)=0$ for the chemical concentration.  From the analytical side, considering the Cauchy problem on the whole space, but in all space dimensions, it is possible to prove the global existence of smooth solutions, if the initial data are small perturbations of a strictly positive constant state, see \cite{thD,DS12}. For the present case of the one-dimensional boundary value problem, when the differential part is linearized, it is possible to prove the global existence and the time asymptotic decay of the solutions, when the initial data consist of small perturbation of stable equilibrium constant states, see \cite{gumanari}. However, the analytical study in the present setting is still undone and is difficult  in  the presence of vacuum, since  the hyperbolic part of the model degenerates as the eigenvalues coincide when the mass density vanishes, but see \cite{masmoudi, xumin} for some rigorous results about the local existence of solutions for related models without chemotaxis. 

 Actually, preliminary numerical simulations show that, even if we start from strictly positive initial data, vacuum may appear in finite time. This occurrence is much more physically relevant with respect to the analogous situation in gas dynamics. If vacuum is not expected to appear really in gases, it is fully relevant when dealing with the density of cells, i.e.: there are admissible regions without cells, and in some sense it is the main goal of a biologically consistent model. This is a situation somewhat similar to what occurs when dealing with  flows in rivers with shallow water type equations \cite{Audusse}. Besides, also looking at the numerical approximation, dealing with vacuum needs for a special care, since we have  to guarantee for the non negativity of the solutions \cite{Bouchut_book}.  This can be understood at the level of the associated numerical flux. A numerical flux resolves the vacuum if for all values of the approximate solution, it is able to generate nonnegative solutions with a finite speed of propagation. In this paper it will be crucial to use schemes with this kind of property.

Another obstacle to a serious numerical exploration is given by the possible lack of proper resolution of nonconstant steady states. Actually, at least for  not too large initial masses, solutions to system (\ref{eq:main_system}) are expected to stabilize for large times, around some global steady states of the system and, when the system reaches equilibrium, the flux is expected to vanish. However, most of the current schemes fail to reproduce this behavior. For instance, a classical centered discretization for the source  is not precise enough 
 near steady states.  
As well known, see for instance \cite{Natalini_Ribot, Gosse_chemo2}, this is a usual problem for schemes dealing with hyperbolic problems with source. This is   why  other approaches have been introduced   to balance properly the fluxes and the sources, so giving a  more accurate approximation at equilibria.
In the case of a semilinear model, obtained from (\ref{eq:main_system}) by neglecting the drift  term $\ds (\rho u^{2})_{x}$ and taking a linear pressure, i.e. $\gamma=1$, Natalini and Ribot  \cite{Natalini_Ribot} have proposed   an Asymptotically High Order method  \cite{Aregba_Briani_Natalini} adapted to the case of a system with an external source term  and set on a
  bounded interval.   This technique  increases the accuracy of the scheme for large times and yields  a  correct asymptotic stabilization near the  nonconstant equilibria. In \cite{Gosse_chemo1, Gosse_chemo2}, Gosse has studied the same problem with the well-balanced technique  \cite{Greenberg_LeRoux, Gosse_Toscani}  in the framework of finite volume schemes, obtaining similar results. However, both techniques are limited to diagonalizable systems and are difficult to extend  to the  quasilinear model (\ref{eq:main_system}). 
 Let us recall also that, in the  case of the quasilinear system with  linear pressure $\gamma=1$ and without considering the presence of vacuum, 
 Filbet and Shu  \cite{Filbet_Shu} have used the Upwinding Sources at Interface (USI) methodology \cite{MR1933816} to obtain a well-balanced scheme. Their main ingredient was the hydrostatic reconstruction, which was originally applied to the Saint-Venant system by Audusse et al. \cite{Audusse}, implying the preservation of steady states with vanishing velocity. Actually, no boundary conditions were considered in  \cite{Filbet_Shu}, and so the scheme was not tested against nonconstant steady states. 

The main goal of this paper is quite different, and consists in introducing  a new scheme which is able to deal both with vacuum problems and with problems  near nonconstant steady states arising from the interaction with the source, which is composed by the friction and the forcing chemotactic terms.  
 The strategy is to adapt the ideas in \cite{Bouchut_Ounaissa_Perthame}, to design a scheme which is able to balance the numerical fluxes with the source term, considering new interface variables and vanishing velocities. Unlike \cite{Filbet_Shu},  we treat the damping term through the  interface variables instead of using additional, fractional steps to integrate it in time. We prove that our scheme  preserves the non negativity of the density and  the stationary states with a vanishing velocity. Finally, the scheme we consider  is able to  treat vacuum states at the free boundary, which was not considered in any of the previous works.

Before we propose and study our scheme, we make a preliminary investigation on nonconstant stationary solutions with vacuum for system (\ref{eq:main_system}) and, in particular,  on  single bump solutions,  which are positive 
     only on one connected region.  We  restrict ourselves mostly  to  the case of a quadratic pressure  $\gamma=2$, which is the simplest one for finding explicit expressions.
     Under these two restrictions, we give a complete description of the stationary solutions with vacuum. We also  numerically show that these solutions are stable and that they can be found as asymptotic states of the system (\ref{eq:main_system}), even in the case of strictly positive initial data.
    Other configurations with several regions of positive densities are also  found numerically  as asymptotic solutions of the system, but we cannot determine for the moment which are the parameters  leading to one configuration or another. These results show that model (\ref{eq:main_system}) gives a realistic description of the vessels formations, which can be successfully tuned against experimental data.

This paper is organized as follows:  in Section \ref{sec:stationary_solutions}, we describe the stationary solutions in the case of $\gamma=2$ and of one region of positive density. Both cases of a lateral bump and of a centered bump are considered. Then, 
in Section \ref{sec:scheme},   we explain the construction of the scheme we use to discretized system (\ref{eq:main_system}) and we prove that it preserves the nonnegativity of the density, and  the stationary states with vanishing velocities. Finally,
in Section \ref{sec:simulations}, we present some numerical results,  studying the accuracy of our scheme and showing that the stationary solutions of  Section \ref{sec:stationary_solutions} are numerically stable. We also  explore the dependence of asymptotic solutions on the parameters of the system, especially on the initial mass and on the value of the adiabatic coefficient $\gamma$.

\section{Stationary solutions}\label{sec:stationary_solutions}
In this section, we analyze the existence and the structure of stationary solutions to system (\ref{eq:main_system}) defined on a bounded interval $[0,L]$ with no-flux boundary conditions given in one space dimension by
\begin{equation}\label{boundary_conditions}
\rho_{x}(0,\cdot)=\rho_{x}(L,\cdot)=0,\quad \rho u(0,\cdot)= \rho u(L,\cdot)=0,\quad \phi_{x}(0,\cdot)=\phi_{x}(L,\cdot)=0.
\end{equation}
Remark first that considering  the evolution problem  (\ref{eq:main_system}) with the previous boundary conditions \eqref{boundary_conditions} implies that the mass of the system is constant in time, namely that
\begin{equation}\label{def_mass}
M=\int_{[0,L]} \rho(x,0) \, dx=\int_{[0,L]} \rho(x,t) \, dx, \, \textrm{ for all } t \geq 0.
\end{equation}
Therefore, the mass $M$ will be considered in what follows as a parameter which characterizes  stationary solutions.
\subsection{Preliminaries}
The steady states of (\ref{eq:main_system}) are the solutions of the following stationary system
\begin{subequations}\label{SS:general}
\begin{equation}
\label{SS1}(\rho u)_x=0,
\end{equation}
\begin{equation}
\label{SS2}\left(\rho u^2+P(\rho)\right)_x=-\alpha\rho u+\chi\rho\phi_x ,
\end{equation}
\begin{equation}
\label{SS3}-D\phi_{xx}=a\rho-b\phi.
\end{equation}
\end{subequations}
Using the boundary condition we have that  
$\rho u=0$, and  equation (\ref{SS2}) becomes 
\begin{equation}\label{SS2_static}
P(\rho)_{x}=\chi\rho\phi_{x},
\end{equation}  
which is automatically satisfied if $\rho$ and $\phi$ are constant. 
More generally,  for a pressure $P(\rho)=\varepsilon\rho^{\gamma}$ with  $\gamma>1$, we have two possible solutions for  equation \eqref{SS2_static}, which are
$$\rho=0$$ 
or 
$$\rho^{\gamma-1}(x)=\frac{\chi(\gamma-1)}{\varepsilon\gamma}\phi(x)+K,$$
 where $K$ is an integration constant.  Remark that functions $\rho$ and  $\phi$ are also constrained to respect  relation  (\ref{SS3}). 
 
 As a consequence,  nonconstant equilibria can be composed of  intervals where the density is strictly positive, which we call bumps, and of intervals where the density vanishes. However, 
 we are not able to determine  a priori  the number of  such intervals as a function of the data. 
  This is why we 
   fix  the number of intervals with positive density  and we denote it by $p\in\mathbf{N}$. In addition, we denote by $x_{i}$ the boundaries of these intervals, setting $x_{0}=0$ and $x_{2p-1}=L$, so that the general form of 
    nonconstant steady  states is
\begin{equation*}\label{SSrho}
\rho(x)=\sum_{k=1}^{p}\rho_{k}(x)\chi_{[x_{2k-2},x_{2k-1}]}(x),
\end{equation*}
where $\chi_{A}$ is the indicator function of the interval $A$.
On  the intervals of the form  $[x_{2k-2},x_{2k-1}]$, the density satisfies the relation 
\begin{equation*}\label{SSrhok}
\rho_k(x)=
\left(\frac{\chi(\gamma-1)}{\varepsilon\gamma}\phi_k(x)+K_k\right)^{\frac{1}{\gamma-1}},
\end{equation*}
where $\phi_{k}$ is the solution to
\begin{equation*}\label{SSphik}
-D\left(\phi_{k}\right)_{xx}=\left\{
\begin{array}{lcl}
a\rho_{k}-b\phi_{k}&\textrm{for}&x\in[x_{2k-2},x_{2k-1}],\ k=1,...,p,\\ \\
-b\phi_{k}&\textrm{for}&x\in[x_{2k-1},x_{2k}], \ k=1,...,p-1.
\end{array}\right.
\end{equation*}
Finding $\phi_{k}$ from the above equations for a general adiabatic exponent $\gamma>1$ is equivalent to solving a second order, nonlinear differential equation of the form
\begin{equation*}\label{IIorderODE}
y_{xx}=A_{1}y+A_{2}y^{\frac{1}{\gamma-1}}+A_{3},
\end{equation*}
where $A_1,A_2,A_3$ are constants. This is a non trivial task apart from the case $\gamma=2$. 

Moreover, a unique  explicit solution can be obtained only in the case of one interval where the density is strictly positive. Otherwise there are more constants to determine than available equations. More precisely, we need to find $4p-2$ constants coming from the integration, $2p-2$ interface points $x$ and $p$ constants $K$. It gives $7p-4$ parameters to determine. On the other hand, we have only $6p-3$ equations coming from the boundary conditions, the continuity of $\phi$ and its first and second space derivatives and from the total mass conservation condition. These two quantities match only  in the case of $p=1$. For $p>1$, we would have an infinite number of stationary conditions depending on $p-1$ parameters. 

 Here, we restrict ourselves to the case of only  one interval where the density is  positive. In this particular case, we can determine exactly the stationary solution  as a function of the initial mass of the density $M$, the length of the domain  $L$ and the other parameters of the system $\alpha$, $\chi$,  $\varepsilon$, $\gamma$, $a$, $b$ and  $D$.  

\subsection{Quadratic pressure $\gamma=2$: the case of a lateral bump} \label{lateral}
We give a precise characterization of the equilibria for $P(\rho)=\varepsilon\rho^2$ when we restrict our attention to solutions with only one bump, i.e.: $p=1$. One of the possible forms is a lateral bump, that is: 
there exists $\bar{x}\in(0,L)$ such that the density satisfies
\begin{equation}\label{rho}
\rho(x)=\left\{\begin{array}{ll}
\ds \frac{\chi}{2\varepsilon}\phi(x)+K,& \textrm{ for } x\in[0,\bar{x}],
\medskip
\\
0, &\textrm{ for } x\in(\bar{x},L],
\end{array}\right.
\end{equation}
and $\phi$  satisfies equation (\ref{SS3}). 
In the following proposition, we completely describe these steady states,  as a function of  the mass $M$ of the density defined in equation  \eqref{def_mass}.
\begin{prop}\label{lat}
Consider system \eqref{SS:general}  on the interval $[0,L]$ with boundary conditions  \eqref{boundary_conditions} and with  a quadratic pressure $P(\rho)=\varepsilon\rho^{2}$.  If $\ds\tau=\frac{1}{D}\left(\frac{a\chi}{2\varepsilon }-b\right)>0$ and  $\ds L>\frac{\pi}{\sqrt{\tau}}$, there exists a unique, positive solution  with a density of the form \eqref{rho} and of  mass $M$. This stationary solution is defined by 
\begin{subequations}\label{lateral_bump}
\begin{equation}\label{bump:rho}
\rho(x)=\left\{\begin{array}{ll}
\ds \frac{\chi}{2\varepsilon}\phi(x)+K,& \textrm{ for } x\in[0,\bar{x}],
\medskip\\
0, &\textrm{ for } x\in(\bar{x},L],
\end{array}\right.
\end{equation}
and
\begin{equation}\label{bump:phi}
\phi(x)=\left\{\begin{array}{ll}
\ds \frac{2\varepsilon b K}{\tau\chi D}\frac{\cos(\sqrt{\tau}x)}{\cos(\sqrt{\tau}\bar{x})}-\frac{aK}{\tau D},& \textrm{ for } x\in[0,\bar{x}],
\medskip\\
\ds -\frac{2\varepsilon  K}{\chi}
\frac{\cosh(\sqrt{\frac{b}{D}}(x-L))}{\cosh(\sqrt{\frac{b}{D}}(\bar{x}-L))}, 
&\textrm{ for } x\in(\bar{x},L].
\end{array}\right.
\end{equation}

The free boundary point $\bar{x}$ is given by 
the only value $\ds \bar{x}\in\frac{1}{\sqrt{\tau}}(\pi/2,\pi)$ such that
\begin{equation}\label{bump:x}
\sqrt{\frac{b}{\tau D}}\tan(\sqrt{\tau}\bar{x})=\tanh(\sqrt{\frac{b}{D}}(\bar{x}-L))
\end{equation}
and the constant $K$ is  equal to
\begin{equation}\label{bump:K}
K=\frac{D}{b}\frac{ M \tau^{3/2}}{\tan(\sqrt{\tau}\bar{x})-\sqrt{\tau}\bar{x}}.
\end{equation}
\end{subequations}
If $\tau<0$, or $\tau>0$ but $\ds L<\frac{\pi}{\sqrt{\tau}}$, then there is no one bump solution to the problem. 
\end{prop}
\begin{proof}
To prove the particular form of the equilibrium (\ref{lateral_bump}), let us first insert (\ref{bump:rho}) into (\ref{SS3}). Assuming $\ds\tau=\frac{1}{D}\left(\frac{a\chi}{2\varepsilon }-b\right)>0$ and using boundary conditions $\ds \phi_{x}(0,\cdot)=\phi_{x}(L,\cdot)=0$,  
 the concentration $\phi$ can be written as
\begin{equation*}
\ds \phi(x)=\left\{\begin{array}{ll}
C_{1}\cos(\sqrt{\tau}x)-\ds \frac{aK}{\tau D}, &  \textrm{ for } x\in[0,\bar{x}],
\medskip\\
C_{2}\cosh(\sqrt{\frac{b}{D}}(x-L)),&  \textrm{ for } x\in(\bar{x},L].
\end{array}\right.
\end{equation*}
Using  the continuity of  functions $\phi$, $\phi_{x}$ and $\phi_{xx}$ at the interface point $\bar{x}$,  we obtain the  values of the constants $C_{1}$ and $ C_{2}$ 
, namely
\begin{equation*}
\displaystyle{C_{1}=\frac{2\varepsilon b K}{\tau\chi D}\cos(\sqrt{\tau}\bar{x})^{-1}},\
\displaystyle{C_{2}=-\frac{2\varepsilon  K}{\chi}\cosh(\sqrt{\frac{b}{D}}(\bar{x}-L))^{-1}} .
\end{equation*}
which 
 yield (\ref{bump:rho}) and (\ref{bump:phi}). 
The remaining equation gives the relation  (\ref{bump:x}) for the location of the interface $\bar{x}$. 
 The parameter $K$ can be calculated from the total  mass  density
\begin{displaymath}
M=\int_{0}^{L}\rho(x)dx=\int_{0}^{\bar{x}}\frac{\chi}{2\varepsilon}\phi(x)dx+K\bar{x}
\end{displaymath}
giving (\ref{bump:K}).

Now we have to determine whether there exists a solution to (\ref{bump:x}) or not. 
The function
$\ds f(x)=\sqrt{\frac{b}{\tau D}}\tan(\sqrt{\tau}x)-\tanh\left(\sqrt{\frac{b}{D}}(x-L)\right)$
is not defined 
on $\ds \frac{1}{\sqrt{\tau}}\left(\frac{\pi}{2}+\pi\mathbb{Z}\right)$ and has a positive derivative elsewhere. 
Its value  $f(0)$ is positive and the sign of $\ds f( \frac{1}{\sqrt{\tau}}\left(\pi+ k \pi\right)), \, k \in \mathbb{Z}, $ is the same as  the sign of $\ds L-\frac{1}{\sqrt{\tau}}\left(\pi+ k \pi\right)$.
So,  if there exists $k \in \mathbb{N}$ such that $\ds L \geq \frac{k \pi }{\sqrt \tau}$, $f$ has exactly one zero  
 in each interval of the form 
$\displaystyle{ \frac{1}{\sqrt \tau}\left((\pi/2,\pi)+\pi \{0,1,...,k-1\}\right)}$. All these zeros are candidates to be a boundary $\bar x$. However,  only the smallest of these points, namely the one belonging to the interval $\ds \frac{1}{\sqrt{\tau}}(\pi/2,\pi),$ guarantees the positivity of the  function $\rho$ defined by \eqref{bump:rho}. Remark that in the case when $\ds L<  \frac{\pi}{\sqrt \tau}$, there is no  stationary solution of the form (\ref{lateral_bump}).

Finally, in the case  $\tau<0$, 
the same  computation leads to  write 
the concentration $\phi$ under  the form
\begin{equation}
\phi(x)=\left\{\begin{array}{ll}
\displaystyle{C_{1}\cosh(\sqrt{-\tau}x)-\frac{aK}{\tau D}},&  \textrm{ for } x\in[0,\bar{x}],\\
\displaystyle{C_{2}\cosh(\sqrt{\frac{b}{D}}(x-L))},& \textrm{ for } x\in(\bar{x},L],
\end{array}\right.
\end{equation}
and the continuity of $\phi,\phi_{x},\phi_{xx}$ at $\bar{x}$, gives
the following relation for the interface point $\bar{x}$ 
\begin{equation}
\displaystyle{\sqrt{-\frac{b}{D\tau}}\tanh(\sqrt{-\tau}\bar{x})=\tanh\left(\sqrt{\frac{b}{D}}(\bar{x}-L)\right)}.
\end{equation}
However, this equation has no solution for $\bar{x}\in(0,L)$, since the left-hand side is negative and  the right-hand side positive.
\end{proof}
\subsection{Quadratic pressure $\gamma=2$ :  case of a centered bump} 
Let us consider the case of a centered bump, namely the case of 
 one interval where the density is strictly positive in the interior of the domain. Denoting by $\bar{x},\bar{y}\in(0,L)$ 
 the boundaries of this interval, with  $\bar{x}<\bar{y}$, the density of the nonconstant steady state can be written  as
\begin{equation}\label{CB:rho}
\rho(x)=\ds \left\{\begin{array}{ll}
0, & \textrm{ for }x\in[0,\bar{x}),\medskip \\
\ds\frac{\chi}{2\varepsilon}\phi(x)+K, &  \textrm{ for } x\in[\bar{x},\bar{y}],\medskip \\
0,&  \textrm{ for } x\in(\bar{y},L].
\end{array}\right.
\end{equation}
In the following proposition, we describe  these stationary solutions as a function of the mass $M$ and of the length of the domain $L$.
 \begin{prop}
Let us consider the system \eqref{SS:general} set on the interval $[0,L]$ with boundary conditions  \eqref{boundary_conditions} and with  a quadratic pressure $P(\rho)=\varepsilon\rho^{2}$. 
 If $\ds\tau=\frac{1}{D}\left(\frac{a\chi}{2\varepsilon }-b\right)>0$  and  $\ds L>\frac{2\pi}{\sqrt{\tau}}$ then there exists a unique, positive solution with a density  of mass $M$ and of the form \eqref{CB:rho}. This solution is given by
 the following expressions:
\begin{subequations}
\begin{equation}\label{centered:rho}
\rho(x)=\ds \left\{\begin{array}{ll}
0, & \textrm{ for }x\in[0,\bar{x}),\medskip \\
\ds\frac{\chi}{2\varepsilon}\phi(x)+K, &  \textrm{ for } x\in[\bar{x},\bar{y}],\medskip \\
0,&  \textrm{ for } x\in(\bar{y},L],
\end{array}\right.
\end{equation}
and
\begin{equation}\label{centered:phi}
\phi(x)=\left\{\begin{array}{ll}
\ds -\frac{2\varepsilon K}{\chi}\frac{\cosh(\ds\sqrt{\frac{b}{D}}x)}{\cosh(\ds \sqrt{\frac{b}{D}}\bar{x})} ,& \textrm{ for } x\in[0,\bar{x}),
\medskip\\
\ds \frac{2\varepsilon b K}{\tau\chi D}\frac{\cos(\sqrt{\tau}(x-\frac{L}{2}))}{\cos(\sqrt{\tau}(\bar{x}-\frac{L}{2}))}-\frac{aK}{\tau D},& \textrm{ for } x\in[\bar{x}, \bar{y}],
\medskip\\
\ds -\frac{2\varepsilon K}{\chi}\frac{\cosh(\ds\sqrt{\frac{b}{D}}(x-L))}{\cosh(\ds\sqrt{\frac{b}{D}}\bar{x})}, 
&\textrm{ for } x\in(\bar{y},L].
\end{array}\right.
\end{equation}
 The free boundary point $\bar{y}$ is such that  $\bar{y}=L-\bar{x}$ and so the solution is symmetric and  the boundary $\bar{x}$ is  given by the only value $\ds\bar{x}\in \left( \frac{L}{2}-\frac{\pi}{\sqrt{\tau}},
 \frac{L}{2}-\frac{\pi}{2\sqrt{\tau}}\right)$ such that
\begin{equation}\label{interface_centered}
\tanh(\sqrt{\frac{b}{D}}\bar{x})=\sqrt{\frac{b}{\tau D}}\tan(\sqrt{\tau}(\bar{x}-L/2)).
\end{equation}
The constant  $K$ is equal to
\begin{equation}\label{K}
 \ds K=\frac{M\tau}{(2\bar{x}-L)\ds\frac{b}{D}-2\ds\sqrt{\frac{b}{D}}\tanh(\sqrt{\frac{b}{D}}\bar{x})}.
\end{equation}
\end{subequations}
For other values of the parameters there are no solutions of this form. 
\end{prop}
\begin{proof}
We follow the same computations as in the case of a lateral bump of Subsection \ref{lateral}. Under the assumption $\tau>0$ and using boundary conditions  \eqref{boundary_conditions},   the concentration $\phi$ can be written as 
\begin{equation*}
\phi(x)=\left\{\begin{array}{ll}
C_{1}\cosh(\sqrt{\frac{b}{D}}x),&\textrm{ for }  x\in[0,\bar{x}),\medskip\\
\ds C_{2}\cos(\sqrt{\tau}x)+C_{3}\sin(\sqrt{\tau}x)-\frac{aK}{\tau D},&\textrm{ for } x\in[\bar{x},\bar{y}],\medskip\\
C_{4}\cosh(\sqrt{\frac{b}{D}}(x-L)),&\textrm{ for } x\in(\bar{y},L].
\end{array}\right.
\end{equation*}
Solving the system of three equations given by the
continuity of $\phi$, $\phi_{x}$ and $\phi_{xx}$ at the interface point $\bar{x}$ (resp. $\bar{y}$) gives
an expression for the three constants $C_{1}$ (resp. $C_{4}$), $C_{2}$ and $C_{3}$ as a function of $\bar x$ (resp. $\bar y$). Therefore, we obtain two expressions  for $C_{2}$ and $C_{3}$, which give us two non linear equations for $\bar x$ and $\bar y$, namely: 
\begin{equation*}
\left\{
\begin{aligned}
 \sin(\sqrt{\tau}\bar{x})\tanh(\sqrt{\frac{b}{D}}\bar{x})-\sin(\sqrt{\tau}\bar{y})\tanh(\sqrt{\frac{b}{D}}(\bar{y}-L))=&\sqrt{\frac{b}{D\tau}}\bigl(\cos(\sqrt{\tau}\bar{y})\\
 &\quad -\cos(\sqrt{\tau}\bar{x})\bigr),\\
 \cos(\sqrt{\tau}\bar{x})\tanh(\sqrt{\frac{b}{D}}\bar{x})-\cos(\sqrt{\tau}\bar{y})\tanh(\sqrt{\frac{b}{D}}(\bar{y}-L))=&\sqrt{\frac{b}{D\tau}}\bigl(\sin(\sqrt{\tau}\bar{x})\\
 &\quad -\sin(\sqrt{\tau}\bar{y})\bigr).
\end{aligned}
\right.
\end{equation*}
This system can be easily rewritten as 
{\small
\begin{equation}\label{equations_x_symmetry}
\left\{
\begin{array}{l}
\ds\tanh(\sqrt{\frac{b}{D}}\bar{x})=\tanh(\sqrt{\frac{b}{D}}(\bar{y}-L))\cos(\sqrt{\tau}(\bar{y}-\bar{x}))-\sqrt{\frac{b}{D\tau}}\sin(\sqrt{\tau}(\bar{y}-\bar{x})),\\
\ds\sqrt{\frac{b}{D\tau}}=\tanh(\sqrt{\frac{b}{D}}(\bar{y}-L))\sin(\sqrt{\tau}(\bar{y}-\bar{x}))+\sqrt{\frac{b}{D\tau}}\cos(\sqrt{\tau}(\bar{y}-\bar{x})).
\end{array}
\right.
\end{equation}
}
From (\ref{equations_x_symmetry}) we get the relation
\begin{displaymath}
\tanh(\sqrt{\frac{b}{D}}\bar{x})=\sqrt{\frac{b}{D\tau}}\frac{\cos(\sqrt{\tau}(\bar{y}-\bar{x}))-1}{\sin(\sqrt{\tau}(\bar{y}-\bar{x}))}=-\tanh(\sqrt{\frac{b}{D}}(\bar{y}-L)),
\end{displaymath}
which implies that $\bar{y}=L-\bar{x}$ with $\bar{x}\in(0,L/2)$. Inserting this relation in equation   (\ref{equations_x_symmetry}), we find that $\bar x$
satisfies equation \eqref{interface_centered}.

Following the same analysis as in  the proof of Prop. \ref{lat}, 
equation \eqref{interface_centered} has a solution iff $\ds L>\frac{2\pi}{\sqrt{\tau}}$ and in the case when $\ds L>\frac{2\pi}{\sqrt{\tau}}$, the only solution $\bar x$   leading to a positive density is  such that $\ds\bar{x}\in \left( \frac{L}{2}-\frac{\pi}{\sqrt{\tau}},
 \frac{L}{2}-\frac{\pi}{2\sqrt{\tau}}\right)$. 

The parameter $K$ is computed thanks to the value of the mass $M$ of the density and is given by equation \eqref{K}.
\end{proof}
\begin{rem}
Computing  solutions $(\bar{x},\bar{y})$ to systems  (\ref{equations_x_symmetry})  is, in general, a non trivial task. In the case of one lateral bump, we found the conditions that guarantee the existence of $\bar{x}$ giving nonnegative density everywhere. Nevertheless, due to the nonlinearity of equation \eqref{bump:x}, we  have to solve it numerically to find the value of  $\bar{x}$. In the case of one centered bump, we find out that the solution is symmetric, which simplifies the computations. 
  The conditions for the existence of solutions can be found 
   and the explicit solution can be calculated. However, the technique we use in that  proof cannot be generalized to  a higher number of bumps. 
\end{rem}
\begin{rem}
Notice that the solution for the lateral bump calculated with $L\slash 2$ and $M\slash 2$ and symmetrized to be defined on the whole interval $[0,L]$ is equal to the one computed for the centered bump. Our computation was essentially aimed to exclude other cases. 
\end{rem}
\section{Numerical approximation}\label{sec:scheme} 

Let us now explain how to construct a reliable scheme in order to perform numerical simulations of system \eqref{eq:main_system}. As a standard guess, we can expect that solutions stabilize on  steady states. This scheme will also enable us to test the stability of the stationary solutions computed in the previous section.

System (\ref{eq:main_system}) couples equations of different natures, i.e. a  quasi-linear system of conservation laws with sources, coupled   with a linear parabolic equation for the evolution of the chemoattractant.  
%
The parabolic part   
can be approximated  using, for example, the classical explicit-implicit Crank-Nicholson method. 

Now, denoting by  $U=(\rho,\rho u)^t$ the vector of the two unknowns, density and momentum,  the   hyperbolic part of system  \eqref{eq:main_system} can be written in the following form
\begin{subequations}\label{eq:hyp_matrix}
\begin{equation}
\label{eq:hyp_matrix_system} U_{t}+F(U)_{x}=S(U),
\end{equation}
where $F$ is the flux function and $S$ the source term, i.e.
\begin{equation}
\label{eq:hyp_matrix_vectors}F(U)=
\left(
\begin{array}{c}
F^{\rho} \\
F^{\rho u}
\end{array}
\right)
=\left(
\begin{array}{c}\rho u
\\
 \rho u^2+P(\rho)
 \end{array}
\right)
 ,\quad S(U)=
 \left(
\begin{array}{c}
 0
\\ 
-\alpha\rho u+\chi\rho\phi_x
\end{array}
\right).
\end{equation}
\end{subequations}
 In this section we present a finite volume scheme for  (\ref{eq:hyp_matrix}) defined on  a  bounded domain  $[0,L]$ with no-flux boundary conditions \eqref{boundary_conditions}.
 The scheme needs  to preserve the non negativity of density and 
   all the steady states of the system. 
\subsection{Well-balanced scheme}
 According to the framework of finite volume schemes, we divide the interval $[0,L]$ into $N$ cells  $C_{i}=[x_{i-1/2},x_{i+1/2})$,  centered at nodes $x_{i}, \, 1\leq i\leq N$. In the following, we will assume, for simplicity, that all the cells have the same length $\ds \Delta x=x_{i+1/2}-x_{i-1/2}$.
   We consider as a semi-discrete approximation of the solution $U$ of system \eqref{eq:hyp_matrix} on cell $C_{i}$ an approximation of the cell average of the solution at time $t>0$, that is to  say  
\begin{equation*}
U_{i}(t)=\frac{1}{\Delta x}\int_{x_{i-1/2}}^{x_{i+1/2}} U(x, t) \, dx.
\end{equation*}
A general semi-discrete, finite volume scheme for (\ref{eq:hyp_matrix}) 
can be defined as 
\begin{equation}\label{eq:scheme_main}
\Delta x\frac{d}{dt}U_{i}(t)+\mathcal{F}_{i+1/2}(t)-\mathcal{F}_{i-1/2}(t)= S_{i}(t),
\end{equation}
where $\ds \mathcal{F}_{i+1/2}(t)$ is an approximation of the flux $F(U(x_{i+1/2},t))$ at the interface point $x_{i+1/2}$ at time $t$ and $\ds S_{i}(t)$  is an approximation of the source term $S(U)$ on the cell $C_{i}$ at time $t$.

A classical choice is to take $\mathcal{F}_{i+1/2}(t)=\mathcal{F}(U_{i}(t),U_{i+1}(t))$, where $\mathcal{F}$ is any consistent $C^{1}$ numerical flux function for the homogeneous problem $\ds  U_{t}+F(U)_{x}=0$. The numerical source is given as $S_{i}(t)=S(U_{i}(t))$, where we  discretized the derivative $\phi_{x}$ with a space centered formula $\ds \phi_{x | C_{i}}\sim \frac{\phi_{i+1}-\phi_{i-1}}{2 \Delta x }$, where  $\phi_{i\pm 1}$ is an approximation of function $\phi$ at  points $x_{i\pm 1}$. 
 However, it is known that this kind of approximation produces large errors near nonconstant steady states. 

Balancing the flux term and the source term 
increases significantly the accuracy near steady states, by imposing an exact discretization of the stationary solutions of the system.
A possible approach 
 is to calculate the flux terms  $\mathcal{F}_{i\pm 1/2}$  in \eqref{eq:scheme_main} as a function of   new  interface variables $U_{i+1/2}^{\pm}$, i.e.
 $\ds \mathcal{F}_{i+1/2}=\mathcal{F}(U_{i+1/2}^{-},U_{i+1/2}^{+})$.
These interface variables will be made precise later on and their computation  will take into account the balance between the flux term and the source.

The technique is  also  to upwind the source term, defined as
$\ds S_{i}=\mathcal{S}_{i+1/2}^{-}+\mathcal{S}_{i-1/2}^{+}$,
in the spirit of the USI method  \cite{MR1933816, Perthame_Simeoni, Katsaounis_Pertham_Simeoni, Bouchut_Ounaissa_Perthame}.
We consider   the following ansatz:
\begin{equation}\label{eq:ansatz}
\mathcal{S}_{i+1/2}^{-}=
\left(\begin{array}{c}
0 \\P\left(\rho_{i+1/2}^{-}\right)-P(\rho_{i})
\end{array}\right), \,
\mathcal{S}_{i-1/2}^{+}=\left(\begin{array}{c}
0 \\P(\rho_{i})-P\left(\rho_{i-1/2}^{+}\right)
\end{array}\right).
\end{equation} 
This ansatz is  motivated by an exact discretization of the stationary part of system \eqref{eq:hyp_matrix_system},   $\ds F(U)_{x}=S(U)$, using that, in the case of a stationary solution,  the momentum $\rho u $ vanishes thanks to boundary conditions.

Therefore, the final scheme can be written as : 
\begin{equation}\label{eq:scheme_main2}
\Delta x\frac{d}{dt}U_{i}+\mathcal{F}(U_{i+1/2}^{-},U_{i+1/2}^{+})-\mathcal{F}(U_{i-1/2}^{-},U_{i-1/2}^{+})=\mathcal{S}_{i+1/2}^{-}+\mathcal{S}_{i-1/2}^{+}.
\end{equation}
We will precise in the following subsection how to reconstruct 
the interface variables $U_{i+1/2}^{\pm}$, which contain  information about the sources.  
This will be done according to the local equilibrium 
in order to make the scheme
consistent with (\ref{eq:hyp_matrix}), preserving the non negativity of the density and preserving the steady states of (\ref{eq:hyp_matrix}) with a vanishing velocity.
\subsection{Reconstruction}
In order to complete the construction of the scheme, we  define the interface variables $U_{i+1/2}^{\pm}$. To satisfy the well-balanced property and increase the accuracy of the approximation near nonconstant steady states,  the reconstruction is obtained from the stationary system
\begin{equation}\label{eq:steady states}
\left\{\begin{array}{l}
(\rho u)_{x}=0,\\
\left(\rho u^{2}+P(\rho)\right)_{x}=-\alpha\rho u+\chi\rho\phi_{x}.
\end{array}\right.
\end{equation}

The system (\ref{eq:steady states}) can be rewritten in terms of the internal energy function $e(\rho)$ which, for a pressure law of isentropic gas dynamics (\ref{eq:pressure_law}), is defined by
$\ds e'(\rho)=\frac{P(\rho)}{\rho^{2}}$.
We  consider the function 
 \begin{equation*}
 \Psi(\rho):=e(\rho)+\frac{P(\rho)}{\rho}= \frac{\varepsilon\gamma}{\gamma-1}\rho^{\gamma-1}, \textrm{ for } \gamma>1,
 \end{equation*}
and we divide the second equation of (\ref{eq:steady states}) by $\rho$, which leads to :
\begin{equation}\label{eq:steady states2}
\left\{\begin{array}{l}
(\rho u)_{x}=0,\\
\ds \left(\frac{u^{2}}{2}+\Psi(\rho)-\chi\phi
\right)_{x}=-\alpha u.
\end{array}\right.
\end{equation}
We integrate now  the previous system  on $[x_{i},x_{i+1/2}]$  (resp.$[x_{i+1/2},x_{i+1}]$) to find  the interface variables 
\begin{equation*}
U_{i+1/2}^{-}=\left(
\begin{array}{c}
\rho_{i+1/2}^{-} \\
\rho_{i+1/2}^{-}u_{i+1/2}^{-}
\end{array}
\right)
\left(\textrm{ resp.  }
U_{i+1/2}^{+}=\left(
\begin{array}{c}
\rho_{i+1/2}^{+} \\
\rho_{i+1/2}^{+}u_{i+1/2}^{+}
\end{array}
\right)\right),
\end{equation*}
and we obtain the two following equations for  $ \rho_{i+1/2}^{-} u_{i+1/2}^{-}$ and $\rho_{i+1/2}^{-}$: 
\begin{subequations}
\begin{equation}\label{eq:integration_momentum}
\rho_{i+1/2}^{-}u_{i+1/2}^{-}=\rho_{i}u_{i}
\end{equation}
and  
\begin{equation}\label{eq:integration_general}
\frac{\varepsilon\gamma}{\gamma-1}\left(\rho_{i+1/2}^{-}\right)^{\gamma+1}+C_{i}\ds (\rho_{i+1/2}^{-})^{2}
+\frac{1}{2}\rho_{i}^2u_{i}^2=0,
\end{equation}
\end{subequations}
using \eqref{eq:integration_momentum}, with
\begin{displaymath}
C_{i}=-\frac{\varepsilon\gamma}{\gamma-1}\rho_{i}^{\gamma-1}+\chi(\phi_{i}-\phi_{i+1/2})-\frac{1}{2}u_{i}^{2}+ \alpha\int_{x_{i}}^{x_{i+1/2}}u(x) \, dx.
\end{displaymath}

For an integer $\gamma>1$,  equation \eqref{eq:integration_general}  is a polynomial of order larger than two. The main difficulty in the reconstruction lies in finding its roots and checking that the form of $U_{i+1/2}^{\pm}$ leads to a consistent scheme that preserves the non negativity of the density. Audusse et al. in \cite{Audusse} introduced the hydrostatic reconstruction for shallow water equations,  assuming  that the velocity is zero at the steady states. This hypothesis  simplifies equation (\ref{eq:integration_general}) such that an explicit solution can be found. This method generates  a scheme, which is well-balanced at equilibria with vanishing velocity. Remark that, for a  system of type  (\ref{eq:main_system}) set on a bounded domain with no-flux boundary conditions  \eqref{boundary_conditions},
 such equilibria are the only possible stationary solutions
  of the system. 

However, in the quasilinear model of chemotaxis (\ref{eq:main_system}), the source in the momentum balance equation contains a damping term together with a chemotaxis term. The assumption $u=0$ at a steady state cancels the friction term  in  the reconstruction and, to satisfy the  consistency property,  this term has to be added separately in the discretization of   (\ref{eq:hyp_matrix}),  namely :
\begin{displaymath}
\Delta x\frac{d}{dt}U_{i}+\mathcal{F}_{i+1/2}-\mathcal{F}_{i-1/2}=\mathcal{S}_{i+1/2}^{-}+\mathcal{S}_{i-1/2}^{+}+\Delta x\left(\begin{array}{c}
0\\
-\alpha\rho_i u_i
\end{array}\right).
\end{displaymath}
This approach was used by Filbet and Shu in \cite{Filbet_Shu} in the case of a linear pressure function, i.e.  $\gamma=1$. In order to include the friction term into the reconstruction,  we define the new interface variables taking into account stationary solutions with constant, instead of vanishing, velocity. This allows to deal with the damping term together with the chemotaxis term and the flux term, but also to simplify   equation (\ref{eq:integration_general}).  
Therefore, considering a velocity  constant  in space and  integrating equations (\ref{eq:steady states2}), we obtain  the two following relations for  $ u_{i+1/2}^{-}$ and $\rho_{i+1/2}^{-}$: 
\begin{subequations}
\begin{equation}\label{reconstruction_u}
u_{i+1/2}^{-}=u_{i}
\end{equation}
and
\begin{equation}\label{reconstruction_rho_partial}
\Psi(\rho_{i+1/2}^{-})-\Psi(\rho_{i})=-\alpha\int_{x_{i}}^{x_{i+1/2}}u(x) \, dx+\chi(\phi_{i+1/2}-\phi_{i}). 
\end{equation}
\end{subequations}
Remark that  for $\gamma>1$ the function $\rho \to \Psi(\rho)$ is strictly increasing and continuous on $[0, +\infty)$ with 	a  finite value at $0$.
So, there exists an inverse function $\Psi^{-1}$, which enables us to find a solution to this last equation.

 It remains now  to explain how to discretize the integral in \eqref{reconstruction_rho_partial} and how to find the approximation $\phi_{i+1/2}$.
The integral in (\ref{reconstruction_rho_partial}) can be discretized by any consistent method,  for example
\begin{displaymath}
-\alpha\int_{x_{i}}^{x_{i+1/2}}u(x) \, dx\approx-\alpha \Delta x (u_{i})_{+}, 
\end{displaymath}
where 
$X_{+}=\max(0,X)$, $X_{-}=\min(0,X)$. 
The computation of $\phi_{i+1/2}$  is not completely plain. The values $\phi_{i}$, which approximate the function $\phi$ at points $x_{i}$, are easily  computed thanks to the parabolic equation for $\phi$ and we use these values to calculate  $\phi_{i+1/2}$.  In order to preserve the non negativity of the density $\rho$, we take $\phi_{i+1/2}=\min(\phi_{i},\phi_{i+1})$. Other choices, as for instance taking the average between $\phi_{i}$ and $\phi_{i+1}$, do not guarantee this property. 

In conclusion,  the reconstruction of the densities becomes
\begin{equation}\label{reconstruction_rho}
\left\{\begin{array}{l}
\rho_{i+1/2}^{-}=\Psi^{-1}\Big(\left[\Psi(\rho_{i})-\alpha \Delta x(u_{i})_{+}+\chi(\min(\phi_{i},\phi_{i+1})-\phi_{i})\right]_{+}\Big),\\
\rho_{i+1/2}^{+}=\Psi^{-1}\Big(\left[\Psi(\rho_{i+1})+\alpha \Delta x (u_{i+1})_{-}+\chi(\min(\phi_{i},\phi_{i+1})-\phi_{i+1})\right]_{+}\Big),
\end{array}\right.
\end{equation}
where the positivity-preserving truncations 
 guarantee the non negativity of $\rho$.

\subsection{Properties of the semi-discrete scheme}
In the following theorem,  we prove some  properties of the  semi-discrete scheme defined by equations (\ref{eq:scheme_main2})-(\ref{eq:ansatz}) with the reconstruction (\ref{reconstruction_u})-(\ref{reconstruction_rho}).
\begin{thm}\label{thm}
Let us consider the system \eqref{eq:hyp_matrix}  set on the interval $[0,L]$ with boundary conditions  \eqref{boundary_conditions}.
Let $\mathcal{F}=(\mathcal{F}^{\rho},\mathcal{F}^{\rho u })^t$ be a consistent, $\mathcal{C}^{1}$  numerical flux preserving the non negativity of $\rho$ for the homogeneous part of system \eqref{eq:hyp_matrix}, $\ds U_{t}+F(U)_{x}=0$. The finite volume scheme \eqref{eq:scheme_main2}-\eqref{eq:ansatz}
with the reconstruction \eqref{reconstruction_u}-\eqref{reconstruction_rho}:
\begin{enumerate}[label=(\roman{*})]
\item is consistent with \eqref{eq:hyp_matrix} away from the vacuum, 
\item  preserves the non negativity of $\rho$, 
\item preserves the steady states given by \eqref{eq:steady states2} with a vanishing velocity.
\end{enumerate}
%
%
\end{thm}
\begin{proof}
$(i)$ To prove the  consistency  of the numerical scheme with  system \eqref{eq:hyp_matrix}, 
we first need to show the consistency of the flux term, i.e.
\begin{displaymath}
\forall i=1,2,...,N\qquad \lim_{U_{i},U_{i+1}\rightarrow U, \, \Delta x \to 0}\mathcal{F}\left(U_{i+1/2}^{-},U_{i+1/2}^{+}\right)=F(U).
\end{displaymath}
This is straightforward using Taylor expansions, since 
\begin{equation*}
 U_{i+1/2}^{-}=U_{i}+O(\Delta x)\textrm{ and }  U_{i+1/2}^{+} =U_{i+1}+O(\Delta x),
\end{equation*}
and, therefore, 
\begin{equation*}
\mathcal{F}\left(U_{i+1/2}^{-},U_{i+1/2}^{+}\right)=\mathcal{F}\left(U_{i},U_{i+1}\right)+O(\Delta x).
\end{equation*}
The consistency  finally comes from the consistency of the numerical flux $\mathcal{F}$  with the analytical flux $F$, i.e. $\ds \mathcal{F}(U,U)=F(U)$.

We have now to prove the consistency of  the discretization of the  source term. To do so,  we use the definition given by Perthame and Simeoni in \cite{Perthame_Simeoni} and  we show that
\begin{displaymath}
\forall i=1,2,...,N\qquad \lim_{U_{i},U_{i+1}\rightarrow U, \, \Delta x\to 0}\frac{1}{\Delta x}\left(\mathcal{S}_{i+1/2}^{-}+\mathcal{S}_{i+1/2}^{+}\right)=S(U).
\end{displaymath} 
Using equations \eqref{eq:ansatz}, we find that  $$
\mathcal{S}_{i+1/2}^{-}+\mathcal{S}_{i+1/2}^{+}
=\left(
\begin{array}{c}
0 \\
P(\rho_{i+1/2}^{-})-P(\rho_{i})+P(\rho_{i+1})-P(\rho_{i+1/2}^{+})\end{array}\right).$$
We use  now Taylor expansions of   $P(\rho_{i+1/2}^{\pm})$. We consider  the density   away from vacuum, so that, for $\Delta x$ small enough,  the positivity-preserving truncations in the reconstruction (\ref{reconstruction_rho}) can be omitted and  we obtain
\begin{align*}
&P(\rho_{i+1/2}^{-})-P(\rho_{i+1/2}^{+})=P(\rho_{i})+\rho_{i}(-\alpha(u_{i})_{+}+\frac{\chi}{2}(\phi_{x,i}-|\phi_{x,i}|))\Delta x\\
&\qquad -P(\rho_{i+1})-\rho_{i+1}(\alpha(u_{i+1})_{-}-\frac{\chi}{2}(\phi_{x,i+1}+|\phi_{x,i+1}|))\Delta x+O(\Delta x^2)
\end{align*}
which proves $(i)$.

$(ii)$  We write the first component of scheme \eqref{eq:scheme_main2} :
\begin{displaymath}
\Delta x\frac{d}{dt}\rho_{i}(t)+\mathcal{F}^{\rho}(U_{i+1/2}^{-},U_{i+1/2}^{+})-\mathcal{F}^{\rho}(U_{i-1/2}^{-},U_{i-1/2}^{+})=0.
\end{displaymath}
 To prove the conservation of the positivity of the density for our scheme, let us show
 that, whenever $\rho_{i}(t)$ vanishes,  the inequality 
 \begin{displaymath}
\mathcal{F}^{\rho}(U_{i+1/2}^{-},U_{i+1/2}^{+})-\mathcal{F}^{\rho}(U_{i-1/2}^{-},U_{i-1/2}^{+})\leq 0,
 \end{displaymath} 
  holds true. Considering separately the two  cases $\phi_{i}>\phi_{i+1}$ and $\phi_{i}<\phi_{i+1}$, we can prove that, if $\rho_{i}=0$, then 
   \begin{equation*}
 \Psi(\rho_{i+1/2}^{-})=\left[\Psi(\rho_{i})-\alpha(u_{i})_{+}\Delta x+\chi(\min(\phi_{i},\phi_{i+1})-\phi_{i})\right]_{+}=0
  \end{equation*} 
  which implies that  $\rho_{i+1/2}^{-}=0$. In the same way, we can prove that  if  $\rho_{i}=0$, then   $\rho_{i-1/2}^{+}=0$.
   Since the numerical flux $\mathcal{F}$  preserves the non negativity of $\rho$ for the homogeneous part of system \eqref{eq:hyp_matrix},  we have
\begin{displaymath}
\mathcal{F}^{\rho}(U_{i},U_{i+1})-\mathcal{F}^{\rho}(U_{i-1},U_{i})\leq 0,
 \end{displaymath}
   whenever  $\rho_{i}(t)$ vanishes, which completes the proof of $(ii)$.
%

$(iii)$. 
We consider a discrete version of the stationary solutions defined by  \eqref{eq:steady states2} with a vanishing velocity, satisfying therefore
\begin{equation}\label{eq:steady states discrete}
\left\{\begin{array}{l}
u_{i}=0,\\
\Psi(\rho_{i+1})-\chi\phi_{i+1}=\Psi(\rho_{i})-\chi\phi_{i}.
\end{array}\right.
\end{equation}
From equations  \eqref{reconstruction_rho} and  \eqref{eq:steady states discrete}, we can see  easily that, in that case,  
$\ds U_{i+1/2}^{-}=U_{i+1/2}^{+}$. 
Stationary solutions \eqref{eq:steady states discrete} are preserved by  scheme \eqref{eq:scheme_main2} iff
\begin{equation*}
\mathcal{F}(U_{i+1/2}^{-},U_{i+1/2}^{+})-\mathcal{F}(U_{i-1/2}^{-},U_{i-1/2}^{+})-\mathcal{S}_{i+1/2}^{-}-\mathcal{S}_{i-1/2}^{+}=0,
\end{equation*}
which is clearly true using the consistency of $\mathcal{F}$ and equations \eqref{eq:hyp_matrix_vectors} and \eqref{eq:ansatz}.
%
%
\end{proof}
\subsection{Properties of the fully discrete scheme}
Let us now consider a time discretization of system \eqref{eq:hyp_matrix} with a given time step $\Delta t$ and  discretization times $\ds t^n=n\Delta t, \, n \in \mathbb{N}$.
Using a standard approximation of the time derivative in scheme \eqref{eq:scheme_main2}, the fully discrete scheme can be written as 
\begin{equation}\label{eq:scheme_fully}
\begin{split}
U_{i}^{n+1}=U_{i}^{n}-\frac{\Delta t}{\Delta x }\bigl(\mathcal{F}(U_{i+1/2}^{n,-},U_{i+1/2}^{n,+})
&-\mathcal{F}(U_{i-1/2}^{n,-},U_{i-1/2}^{n,+})
\bigr)\\
& +\frac{\Delta t}{\Delta x }\left(\mathcal{S}_{i+1/2}^{n,-}+\mathcal{S}_{i-1/2}^{n,+}\right),
\end{split}
\end{equation}
where $U_{i}^{n}$ is an approximation of  the solution $U$ of system \eqref{eq:hyp_matrix} on cell $C_{i}$  at time $t^n$,  $U_{i+1/2}^{n, \pm}$ are  the values of the  interface variables at time $t^n$, and
 \begin{equation}\label{eq:ansatz_full}
\mathcal{S}_{i+1/2}^{n,-}=
\left(\begin{array}{c}
0 \\P\left(\rho_{i+1/2}^{n,-}\right)-P(\rho_{i}^n)
\end{array}\right), \,
\mathcal{S}_{i-1/2}^{n,+}=\left(\begin{array}{c}
0 \\P(\rho_{i}^n)-P\left(\rho_{i-1/2}^{n,+}\right)
\end{array}\right).
\end{equation} 
We can easily see  that the time integration preserves two of the properties proved in Theorem \ref{thm}, namely 
the consistency  and the conservation of the stationary solutions.  However,  we have to find a suitable  stability condition for the scheme, that is to say the relation between  the space step and the time step to preserve the non negativity of the density. To establish it, we use the notion of invariant domain by interface given in \cite{Bouchut_book} and 
we follow the proof presented in \cite{Audusse}. Indeed,  proving directly that the scheme \eqref{eq:scheme_fully} preserves a convex domain, such as the positive half-plane,   is a hard task since the stencil of scheme \eqref{eq:scheme_fully} is composed of three points. To simplify the computations, we consider the weaker notion of preservation of the domain by interface which  consists in  proving two inequalities involving two points each. However, it is proved in   \cite{Bouchut_book} that these two notions are equivalent under  a slightly more restrictive stability condition linking the time step, the space step and a numerical velocity.

We first  give the definition of a solver  preserving the non negativity by interface for a fully discrete scheme, in the case of a homogeneous system. 
\begin{defn}\label{def}
A solver $\ds\mathcal{F}=(\mathcal{F}^{\rho},\mathcal{F}^{\rho u })^t$ for the  homogeneous system $\ds U_{t}+F(U)_{x}=0 $  preserves the non negativity of $\rho$ by interface with a numerical speed $\sigma(U_{i}^n,U_{i+1}^n)\geq 0$ if
whenever the CFL stability condition 
\begin{displaymath}
\sigma(U_{i}^n,U_{i+1}^n)\Delta t\leq\Delta x
\end{displaymath}
holds, we have
\begin{equation*}
\left\{\begin{array}{l}
\rho_{i}^n-\ds\frac{\Delta t}{\Delta x}\left(\mathcal{F}^{\rho}(U_{i}^n,U_{i+1}^n)-\rho_{i}^n u_{i}^n\right)\geq 0,\medskip\\
\rho_{i+1}^n-\ds\frac{\Delta t}{\Delta x}\left(\rho_{i+1}^n u_{i+1}^n-\mathcal{F}^{\rho}(U_{i}^n,U_{i+1}^n)\right)\geq 0.
\end{array}  \right.
\end{equation*} 
\end{defn}
Let us assume that the numerical flux we use for the flux discretization preserves non negativity by interface. The following proposition gives the stability condition to conserve this property in the case  of  system \eqref{eq:hyp_matrix} with source term.
\begin{prop}
Let us assume that the homogeneous flux $\mathcal{F}$ preserves the non negativity of $\rho$ by interface. Then the fully discrete scheme
\eqref{eq:scheme_fully}
with the reconstruction at interfaces given by  \eqref{reconstruction_u}-\eqref{reconstruction_rho} preserves the non negativity of $\rho$ by interface, which means than whenever the CFL  stability  condition
\begin{equation}\label{WB:prop_fully_discrete_CFL}
\sigma(U_{i+1/2}^{n,-},U_{i+1/2}^{n,+})\Delta t\leq\Delta x
\end{equation}
holds, we have
\begin{equation}\label{WB:prop_fully_discrete1}
\left\{\begin{array}{l}
\rho_{i}^n-\ds\frac{\Delta t}{\Delta x}\left(\mathcal{F}^{\rho}(U_{i+1/2}^{n,-},U_{i+1/2}^{n,+})-\rho_{i}^n u_{i}^n\right)\geq 0,
\medskip
\\
\rho_{i+1}^n-\ds\frac{\Delta t}{\Delta x}\left(\rho_{i+1}^n u_{i+1}^n-\mathcal{F}^{\rho}(U_{i+1/2}^{n,-},U_{i+1/2}^{n,+})\right)\geq 0.
\end{array}\right.
\end{equation} 
\end{prop}
\begin{proof}
Assuming   the CFL  stability  condition (\ref{WB:prop_fully_discrete_CFL}) holds,  we have 
\begin{equation*}
\left\{\begin{array}{l}
\ds\rho_{i+1/2}^{n,-}-\frac{\Delta t}{\Delta x}\left(\mathcal{F}^{\rho}(U_{i+1/2}^{n,-},U_{i+1/2}^{n,+})-\rho_{i+1/2}^{n,-}u_{i+1/2}^{n,-}\right)\geq 0,
\medskip\\
\ds \rho_{i+1/2}^{n,+}-\frac{\Delta t}{\Delta x}\left(\rho_{i+1/2}^{n,+}u_{i+1/2}^{n,+}-\mathcal{F}^{\rho}(U_{i+1/2}^{n,-},U_{i+1/2}^{n,+})\right)\geq 0,
\end{array}  \right.
\end{equation*} 
thanks to the Definition \ref{def} of the  preservation of  the non negativity for the homogeneous system. 
It is equivalent to 
\begin{equation}\label{WB:prop_fully_discrete2}
\left\{\begin{array}{l}
\displaystyle{\left(1+\frac{\Delta t}{\Delta x}u_{i}^n\right)\rho_{i+1/2}^{n,-}-\frac{\Delta t}{\Delta x}\mathcal{F}^{\rho}(U_{i+1/2}^{n,-},U_{i+1/2}^{n,+})\geq 0,}
\medskip\\
\displaystyle{\left(1-\frac{\Delta t}{\Delta x}u_{i+1}^n\right)\rho_{i+1/2}^{n,+}+\frac{\Delta t}{\Delta x}\mathcal{F}^{\rho}(U_{i+1/2}^{n,-},U_{i+1/2}^{n,+})\geq 0.}
\end{array} \right.
\end{equation}
From the reconstruction (\ref{reconstruction_rho})  we know that $\rho_{i+1/2}^{n,-}\leq\rho_{i}^n$ and $\rho_{i+1/2}^{n,+}\leq\rho_{i+1}^n$. So, as long as
\begin{displaymath}
\left(1+\frac{\Delta t}{\Delta x}u_{i}^n\right)\geq 0 \textrm{ and } \left(1-\frac{\Delta t}{\Delta x}u_{i+1}^n\right)\geq 0,
 \end{displaymath} 
conditions (\ref{WB:prop_fully_discrete2}) imply  equations (\ref{WB:prop_fully_discrete1}). 

Since $|u_{i}^n|\leq\sigma(U_{i+1/2}^{n,-},U_{i+1/2}^{n,+})$ for all $i=1,2,...,N$, under the CFL condition (\ref{WB:prop_fully_discrete_CFL}), we have  
 $\ds u_{i}^n\geq-\frac{\Delta x}{\Delta t}$ and $\ds u_{i+1}^n\leq\frac{\Delta x}{\Delta t}$, 
 which ends the proof.
\end{proof}

\section{Numerical results}\label{sec:simulations}
In this section, we analyze numerically 
the asymptotic behavior of system (\ref{eq:main_system}) 
 using the 
 scheme introduced in the previous section. First, we compare three different Riemann solvers, Roe, HLL and Suliciu and we show that Suliciu solver is the most adapted to treat vacuum. 
 Then,  in order to study the accuracy of our  scheme, we compare it with a standard  finite difference method with centered in space discretization of the source term. We will see that our  scheme captures better the interface with vacuum, 
 shows less diffusion than the standard finite difference method and gives a proper resolution of nonconstant steady states.  
 Therefore, 
 for  the following numerical simulations, we use 
  the Suliciu relaxation solver, described in \cite{Bouchut_book} for the system of isentropic gas dynamics, with an upwinding of the source term. 
 
This scheme is accurate enough to 
 study numerically
  the stability of the lateral bump and the  dependence of the  asymptotic solutions of system (\ref{eq:main_system}) on some of the   parameters of the system. In particular, for $\gamma=2$, we are interested in 
   the asymptotic number of  bumps 
  for different lengths of the domain $L$ and values of the chemotactic sensitivities $\chi$. We also compare the behavior of the system for different values of the adiabatic exponent $\gamma$ and, finally,  we study the influence of the initial mass  on the structure of  asymptotic equilibria, comparing the results for $\gamma=2$ and $\gamma=3$.  
 
 \subsection{Comparison of different solvers for the homogeneous part}

In order to obtain a numerical approximation using a finite volume scheme we have to calculate numerical fluxes between control cells. This fluxes evaluator is based on solving the Riemann problem at each facet in exact or approximate form. Hyperbolic, homogeneous part of the model (\ref{eq:main_system}) coincides with the isentropic gas dynamics system for which many Riemann solvers are available. However, due to the presence of the source term some of analytical properties of solutions are modified leading to numerical difficulties that cannot be handle by all known Riemann solvers. In particular, occurrence of vacuum and asymptotic, nonconstant states may cause instabilities and negative values of the density. This is the reason why we compare different approximate Riemann solvers, Roe's method, HLL and Suliciu solvers and explain our choice of Suliciu as the most adapted for numerical analysis in the following sections. 

In the first simulation we consider system (\ref{eq:main_system}) with $\varepsilon=D=a=b=L=1$, $\chi=50$,  a quadratic pressure $\gamma = 2$ and the initial data defined as follows
\begin{displaymath}
\rho_{0}(x)=1+\sin(4\pi|x-L\slash 4|),\quad\phi_{0}=0,\qquad u_{0}=0.
\end{displaymath}
Figure~\ref{fig:Comparison_error} presents the  $L^{2}$  (on the left)  and $L^{\infty}$  (on the right) numerical errors for the three previously mentioned solvers with well-balancing of the source term. Behavior at two different grid sizes $\Delta x=0.05$ and $\Delta x=0.01$ is studied. The reference solution is obtained using Suliciu solver and well-balancing reconstruction on the fine grid with mesh size $\Delta x=10^{-3}$. The time step satisfies $\ds  \Delta t=0.9 \Delta x/ \lambda$.
We observe similar behavior of the errors for all the solvers. They oscillate at the beginning of the evolution and stabilize with time. There is no significant difference between them, although the errors for  Roe's method seem a little bigger than in the case of HLL and Suliciu. 

 \begin{figure}[ht]
\begin{center}
\begin{tabular}{cc}
\includegraphics[scale=0.17]{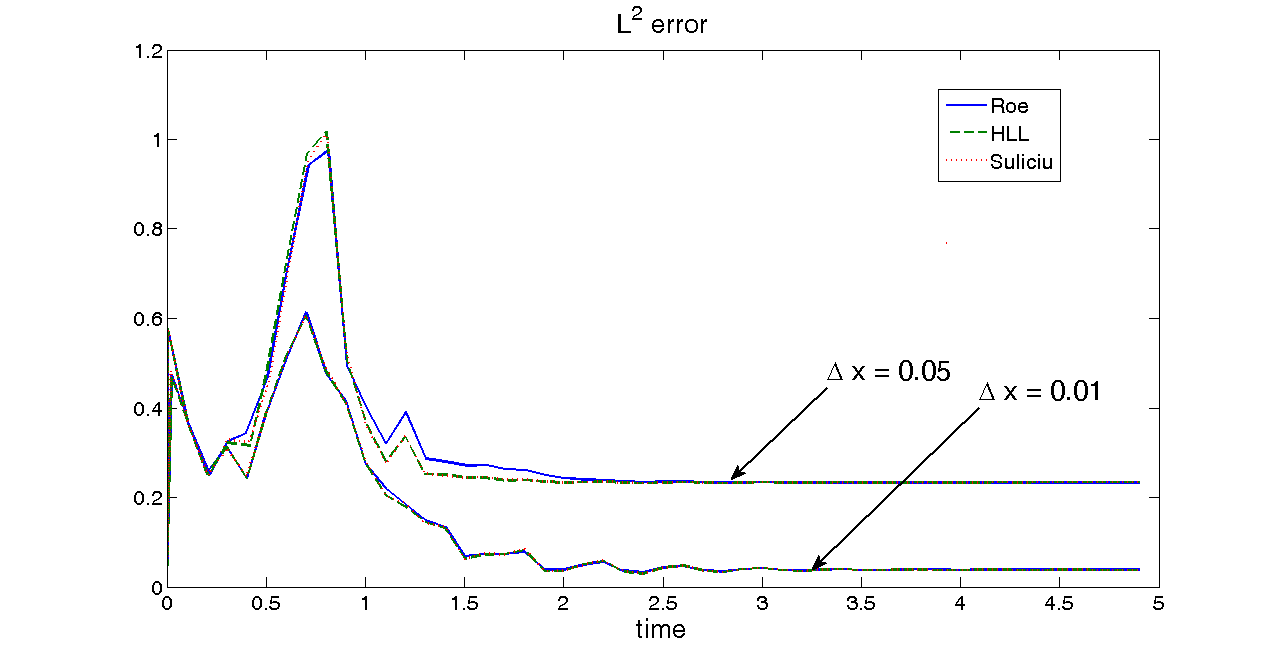}&
\includegraphics[scale=0.17]{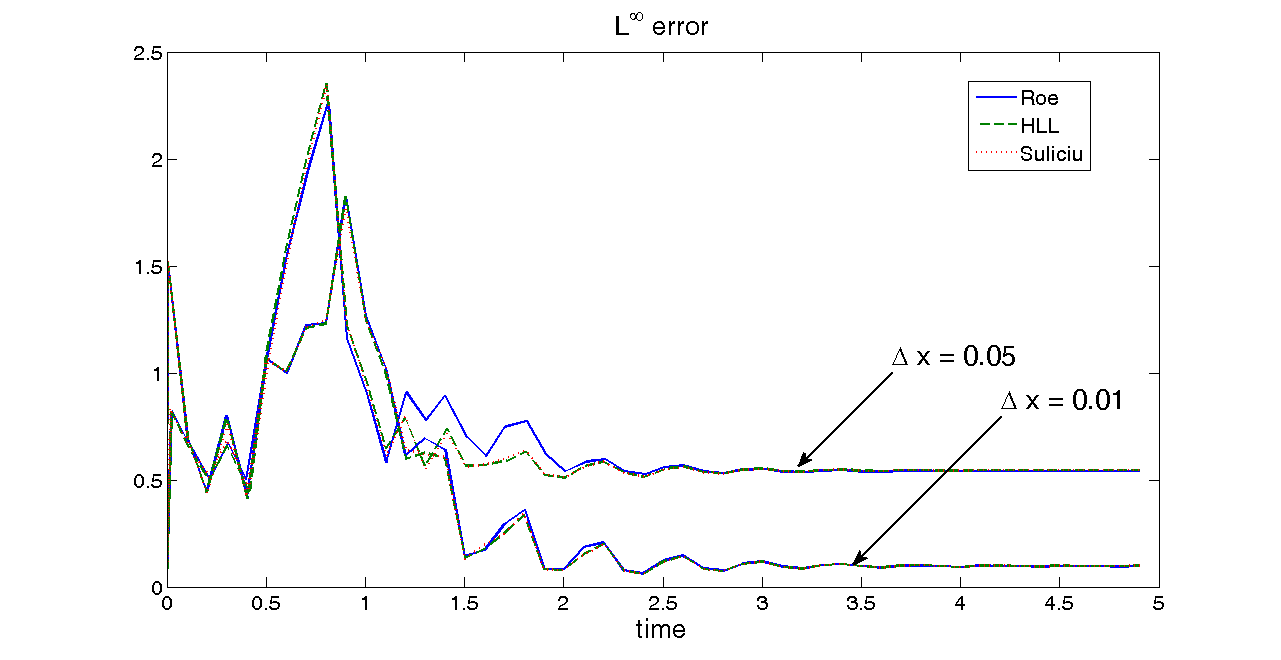}
\end{tabular}
\caption{Time evolution of $L^{2}$ (on the left) and $L^{\infty}$  (on the right) numerical errors of the density $\rho$ for system  \eqref{eq:main_system}  in the case $P(\rho)=\varepsilon\rho^{2}$, with $\varepsilon=D=a=b=L=1$, $\chi=50$,  approximated using a finite volume, well-balanced scheme (\ref{eq:scheme_fully}).  Three different  Riemann solvers: Roe method (in solid blue line), HLL solver (in dashed green line), Suliciu relaxation solver  (in dotted red line) are compared for $\Delta x=0.05$ and $\Delta x=0.01$ .}
\label{fig:Comparison_error}
\end{center}
\end{figure}  

In the second test, we increase the adiabatic coefficient $\gamma$ in the pressure function taking $\gamma=3$. Figure~\ref{fig:Comparison_gamma3} presents the density profiles at asymptotic states containing vacuum and steep gradients near the interfaces between regions where the density is strictly positive and regions where the density vanishes. We observe that  Roe's method, which is based on a linearization of the system, fails and produces negative values of the density near vacuum. We mention here that the classical Stager-Warming flux splitting also oscillates at the interface with vacuum. The two other solvers, HLL and Suliciu, are stable, they preserve non negativity of the density and approximate the interface with high resolution. Suliciu's approach is to use   a relaxation scheme in which the mass conservation equation is not relaxed. This makes it less diffusive than the HLL solver and allows to capture better contact discontinuities. Moreover, in general, the HLL solver needs very careful wave speed estimates. However, in the case of the isentropic gas equations, it is not clear a priori that there is a difference between these two solvers and, in our tests, we don't observe any significant distinction between these two methods. As our model contains source terms leading to steep gradients of the density and appearance of regions where the density vanishes, we decide to use Suliciu solver, which is  adapted to treat  vacuum. \newline

\begin{figure}[ht]
\begin{center}
\begin{tabular}{ccc}
\includegraphics[scale=0.11]{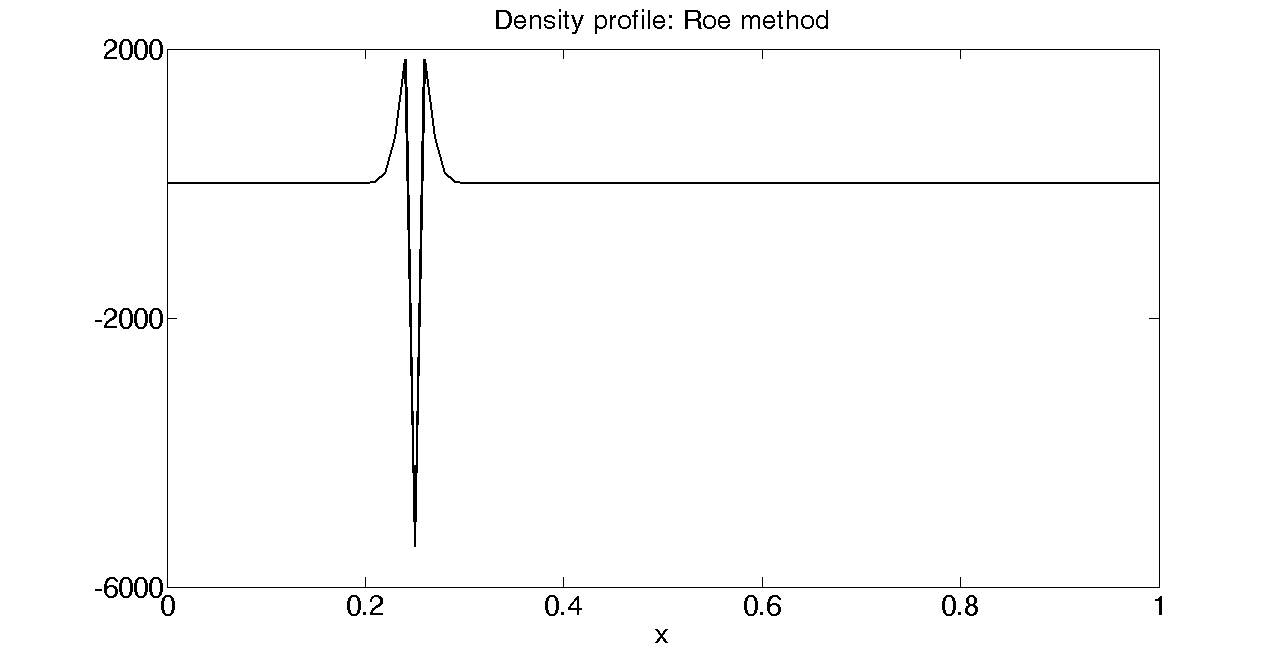}&
\includegraphics[scale=0.11]{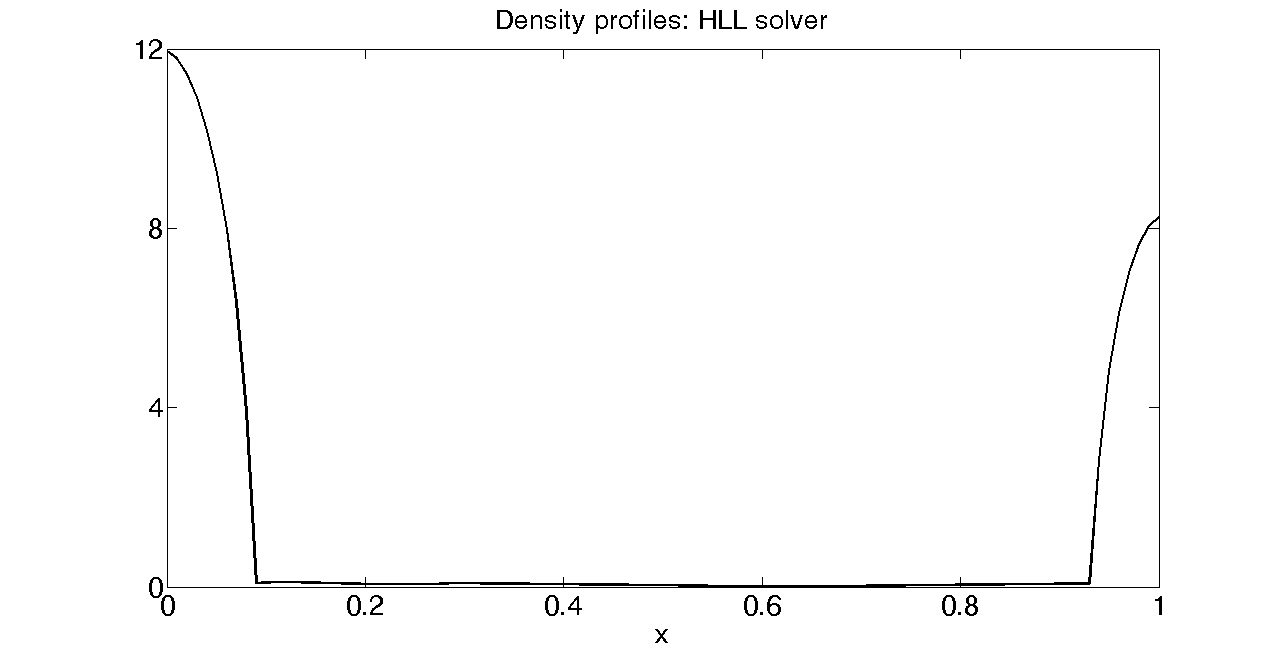}&
\includegraphics[scale=0.11]{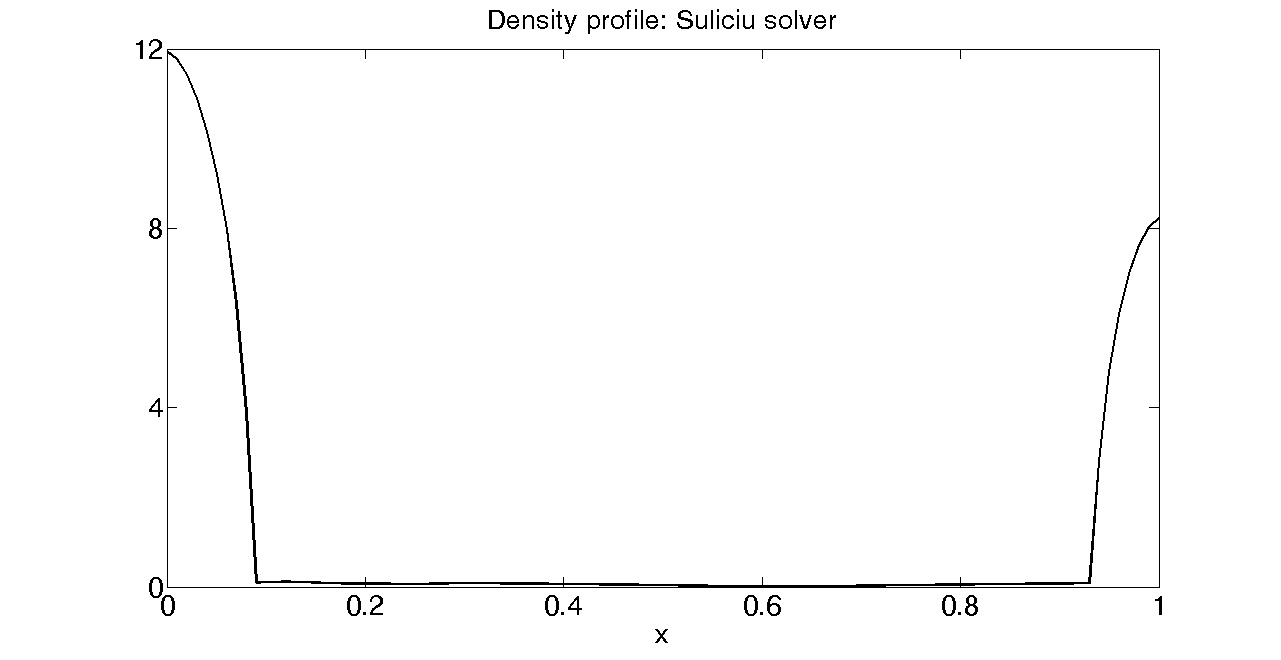}\\
Roe&HLL&Suliciu
\end{tabular}
\caption{Density profiles  for system  \eqref{eq:main_system}  at time $T=5$  in the case $P(\rho)=\varepsilon\rho^{3}$, with $\varepsilon=1$,$D=0.1$, $a=20$, $b=10$, $\chi=10$, $M=1$  approximated using a finite volume, well-balanced scheme (\ref{eq:scheme_fully}).  Three different approximate Riemann solvers: Roe method (on the left), HLL solver (in the middle),  Suliciu relaxation solver (on the right) are compared. }
\label{fig:Comparison_gamma3}
\end{center}
\end{figure}  

\begin{rem}
Numerical simulations indicate that none of the three approximate Riemann solvers that we studied is stable for  $\gamma>2$ with  large initial masses.
 Therefore, our numerical analysis of the system (\ref{eq:main_system}) is performed for initial masses small enough to assure the stability of the scheme.    
\end{rem}

\subsection{Accuracy of the numerical approximation for the source term}
Now,  we analyze how the finite volume numerical scheme \eqref{eq:scheme_fully}-(\ref{eq:ansatz_full})
with  the reconstruction (\ref{reconstruction_u})-(\ref{reconstruction_rho}) manages to capture a particular 
asymptotic behavior of system (\ref{eq:main_system}). More precisely, we show that the finite volume approach with an upwinding of the source term behaves 
  better   near nonconstant steady states 
than a classical finite difference  centered discretization. 
In this subsection, we consider the case of a lateral bump for  $\gamma=2$. 

Figure~\ref{fig:Comparison_source} 
displays asymptotic density profiles  at time $T=100$ obtained with various schemes - finite volume with upwinding, finite volume without upwinding and finite difference without upwinding - 
and the exact solution given by (\ref{lateral_bump}). On the left, we show  the computations  with a space step $\Delta x$ equal to $0.05$ and on the right, equal to $0.01$. The initial data are the same as in the previous subsection and the time step satisfies the same stability condition $\ds  \Delta t=0.9 \Delta x/ \lambda$.
We see that
  the finite volume scheme with the well-balanced property gives clearly the most accurate 
   location of the interface.  Indeed, studying the different approximations of the source term, we notice that the centered discretization  produces a bigger error at the steady state. Then, comparing finite difference and finite volume, we observe that the finite difference method is characterized by a very high numerical diffusion, in contrast to the finite volume approach. However, decreasing the space step,  all the schemes converge to the reference solution, even if the finite volume well-balanced scheme is still the most accurate to give the correct location of the free boundary.
\begin{figure}[htbp!]
\begin{center}
\begin{tabular}{cc}
\includegraphics[scale=0.17]{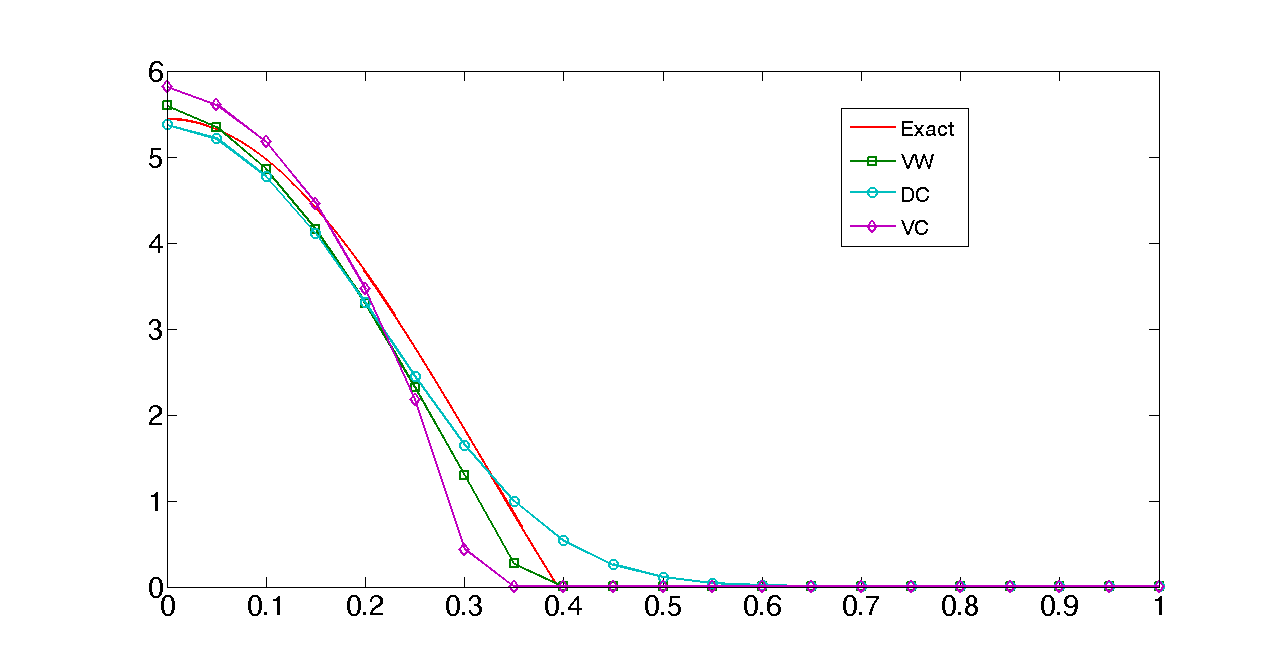}&
\includegraphics[scale=0.17]{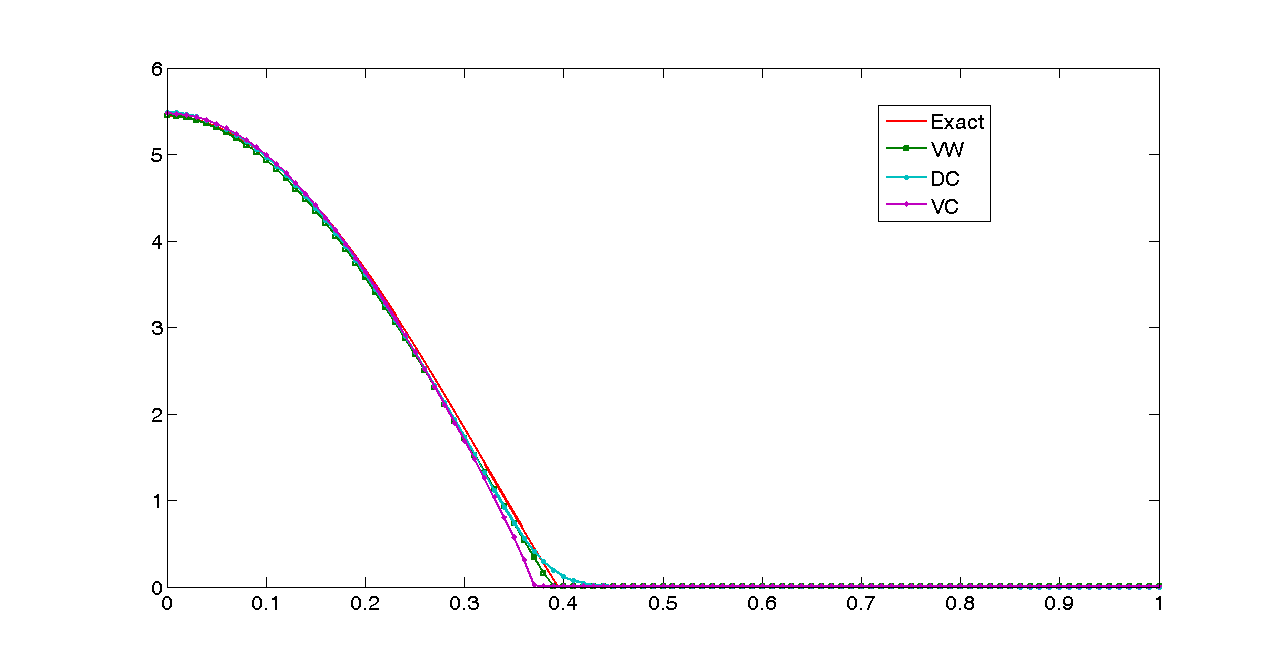}\\
$\Delta x=0.05$&$\Delta x=0.01$
\end{tabular}
\caption{{\small{Density profiles for system (\ref{eq:main_system}) with $P(\rho)=\rho^{2}$, total mass $M=1$ and $L=1$. 
The other parameters of the system are
$D=0.1$, $a=20$, $b=10$, $\chi=10$. Comparison between different methods of approximation of a source term: (green) VW - finite volume method with well-balanced reconstruction; (pink)  VC - finite volume method with centered in space discretization of the source term, (blue) DC -  finite difference scheme with centered in space approximation of the source. Reference solution (red) is given by the exact solution (\ref{lateral_bump})}}}
\label{fig:Comparison_source}
\end{center}
\end{figure}

Another important point in  the numerical approximation of the solutions of system (\ref{eq:main_system}) defined on a bounded domain with no-flux boundary conditions is the accuracy  of the velocity. In this case,  the momentum should vanish at  steady states. Figure~\ref{fig:Comparison_velocity} presents the asymptotic  states of the density and the momentum obtained by two different methods : on top,  the finite volume  well-balanced scheme 
and  on bottom, the finite difference scheme with a centered discretization of the source term. On the right, we can see the corresponding residues of the momentum. First, we observe that at time $T=100$ the residues are less than  $10^{-8}$ and decreasing, which means that the steady state is  reached. Then,  we see that the approximation of the density is comparable for both schemes, which was 
also observed in the previous test. However, the finite difference approach produces an error in the momentum profile. Its $L^{\infty}$ norm is close to one instead of being equal to zero. 
\begin{figure}[htbp!]
\begin{center}
\begin{tabular}{cc}
\includegraphics[scale=0.17]{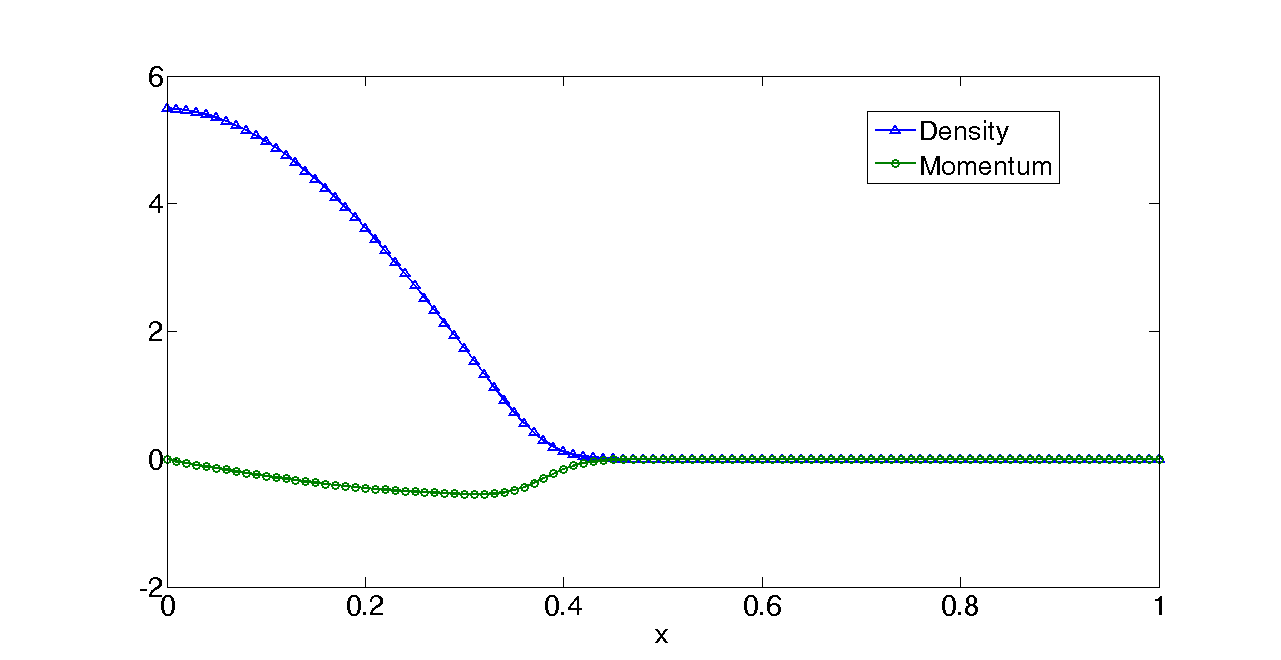}&\includegraphics[scale=0.17]{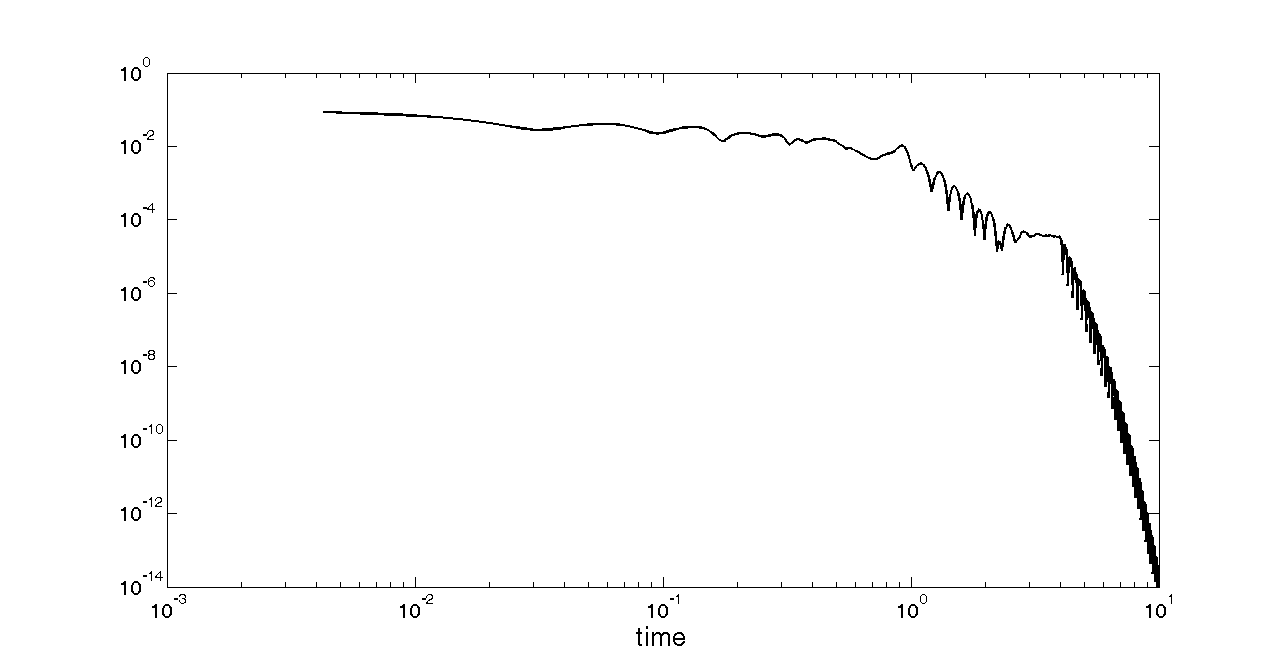}\\
\includegraphics[scale=0.17]{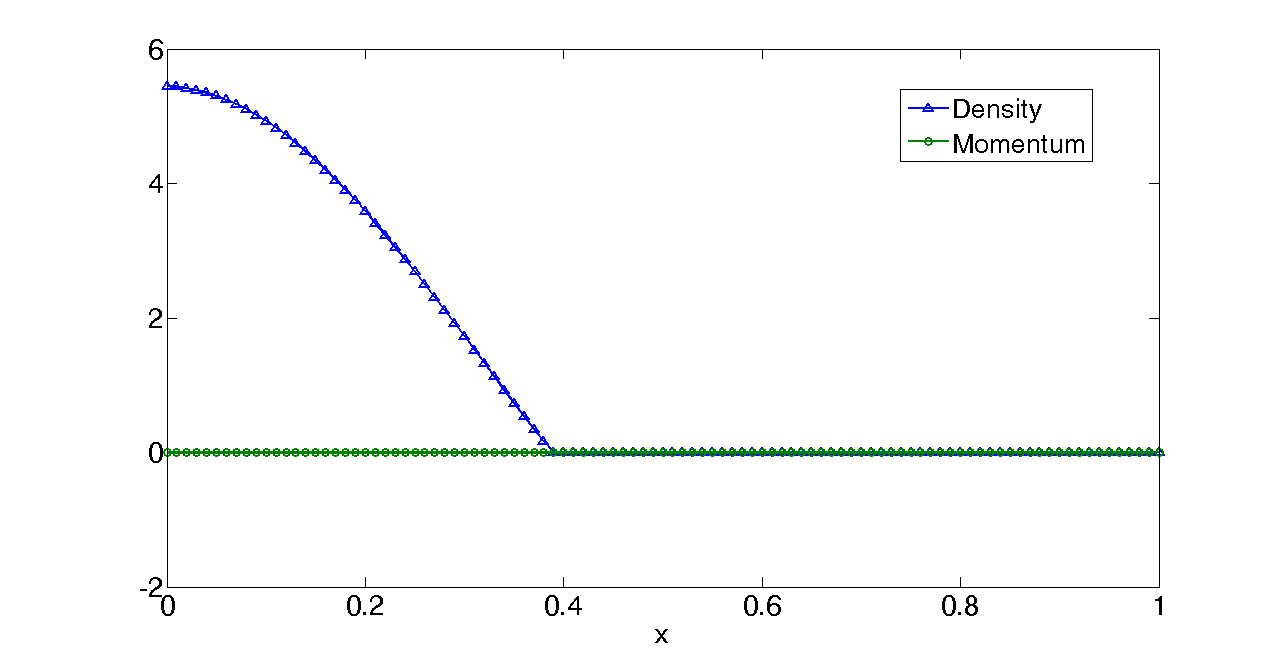}&\includegraphics[scale=0.17]{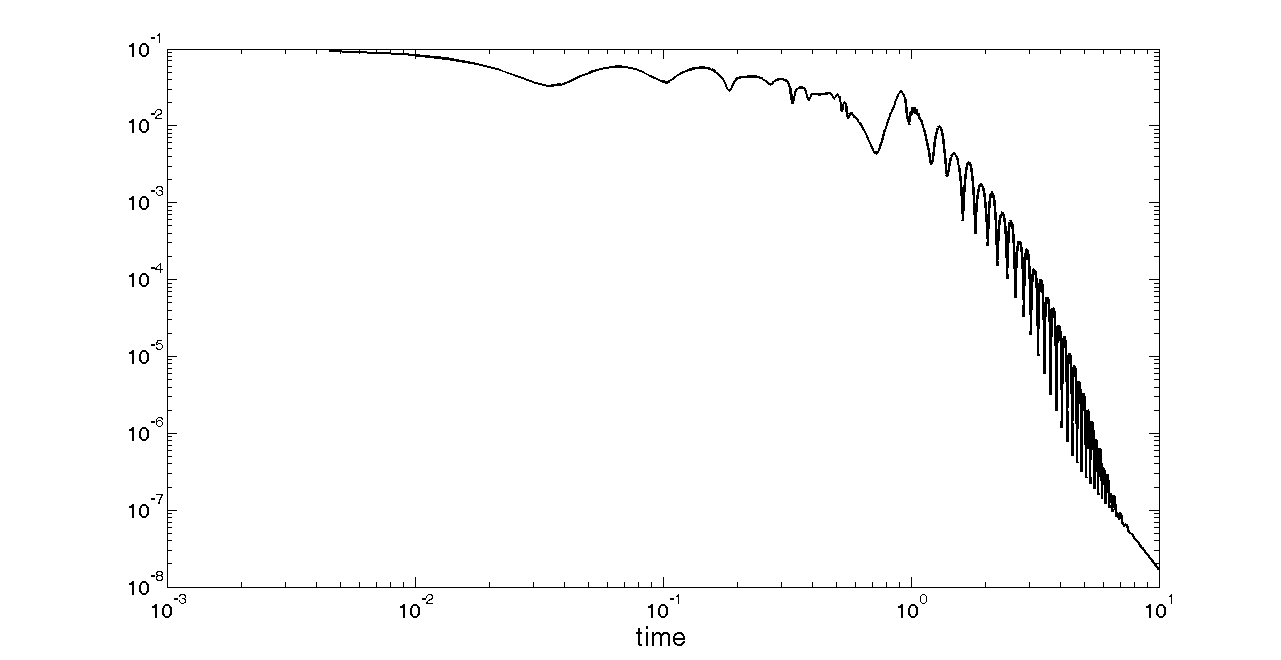}
\end{tabular}
\caption{{\small{On the left:  Asymptotic density (blue) and momentum (green) profiles 
 for system (\ref{eq:main_system})  in the case $P(\rho)=\varepsilon\rho^{2}$ and  $\varepsilon=D=a=b=L=1$, $\chi=50$ obtained by: (on the top) a finite difference scheme with centered in space approximation of the source ; (on the bottom) a finite volume, well-balanced scheme. On the right:  the corresponding residues of the momentum.}}}
\label{fig:Comparison_velocity}
\end{center}
\end{figure} 

The two previous subsections justify the choice of a finite volume well-balanced scheme with a Suliciu solver adapted to treat vacuum. This scheme is accurate enough to study the stability of the stationary solutions  and the asymptotic behavior of the system and we use it in the following  subsections. In what follows, the results are displayed at time $T=100$ with a space step equal to $0.01$ and a time step satisfying the previous stability condition.

\subsection{Stability of a lateral bump}
The first issue we study  is the stability of the stationary solutions computed exactly in subsection \ref{sec:stationary_solutions}. We consider here system (\ref{eq:main_system}) with a quadratic pressure function $P(\rho)=\varepsilon\rho^{2}$ and  with the following parameters:  $\varepsilon=D=a=b=L=1$ and $\chi=50$.   This choice guarantees that  $\ds\tau=\frac{a\chi}{2\varepsilon D}-\frac{b}{D}>0$ and  $\ds L>\pi\slash\sqrt{\tau}$ so the nonconstant steady state with one lateral bump  exists. 
The initial  mass  is taken equal to $\ds M_{0}=1+\frac{1}{\pi}$.
Assuming that the initial datum is not symmetric,  the equilibrium has the form of the lateral bump (\ref{lateral_bump}) with the interface point $\bar{x}\approx 0.3943$. This value is  obtained by solving numerically (\ref{bump:x}).

We analyze the stability of this steady state under two different types of perturbations. First, we perturb the location of the interface point $\bar{x}$. More precisely, we take  initial data of the form (\ref{lateral_bump}), in which $\bar{x}$ is replaced by $x^*=\bar{x}+\delta$. 
We also recalculate the parameter $K$,  the density $\rho$ and the concentration $\phi$ in order to have a perturbation with a zero mass.
 The parameter $\delta$ cannot be too large, since the definition (\ref{lateral_bump}) would lead to negative values of the density. This type of perturbation 
 modifies the solution on its whole support.  In the second test,  we change only the density profile in a small region  $[x_{1},x_{2}]$ satisfying $0<x_{1}<x_{2}<\bar{x}$ such that the initial density is defined as follows : 
\begin{equation}\label{perturb}
\rho_{0}(x)=\left\{\begin{array}{ll}
\rho(x),& \textrm{ for }0<x\leq x_{1},\\
\rho(x_{1}),&\textrm{ for }x_{1}<x\leq x^*,\\
\rho(x_{2}),&\textrm{ for }x^*<x\leq x_{2},\\
\rho(x),&\textrm{ for }x_{2}<x\leq\bar{x},\\
0,&\textrm{ for }\bar{x}<x\leq L,
\end{array}\right.
\end{equation}
where $\rho$ is the exact solution given by (\ref{lateral_bump}). The location of the jump $x^{*}\in(x_{1},x_{2})$ is chosen such that the mass of the perturbation is zero.

\begin{figure}[htbp!]
 \begin{center}
 \begin{tabular}{cc}
\includegraphics[scale=0.17]{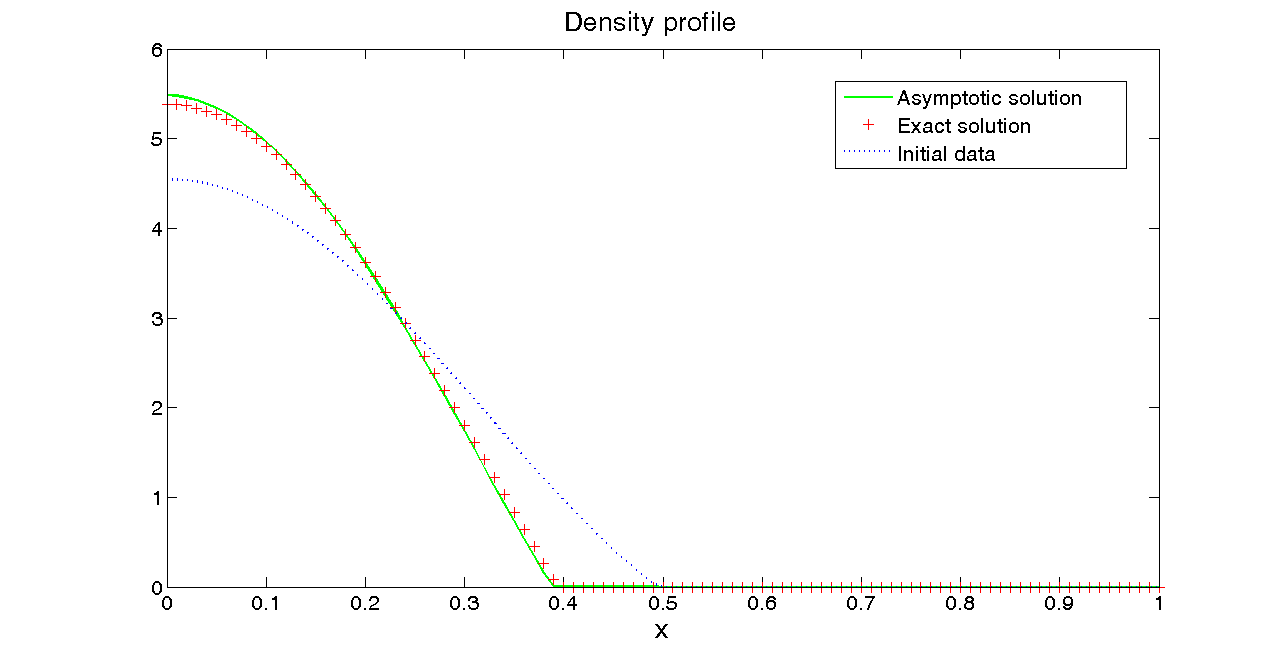}&\includegraphics[scale=0.17]{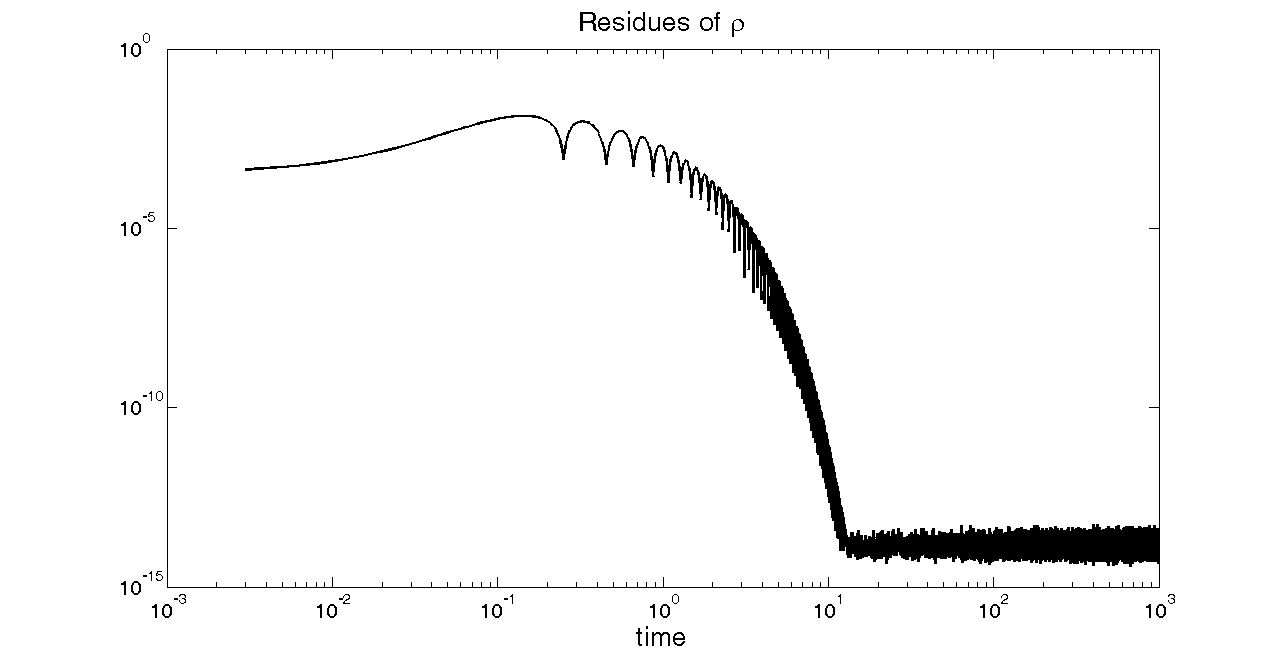}\\
\includegraphics[scale=0.17]{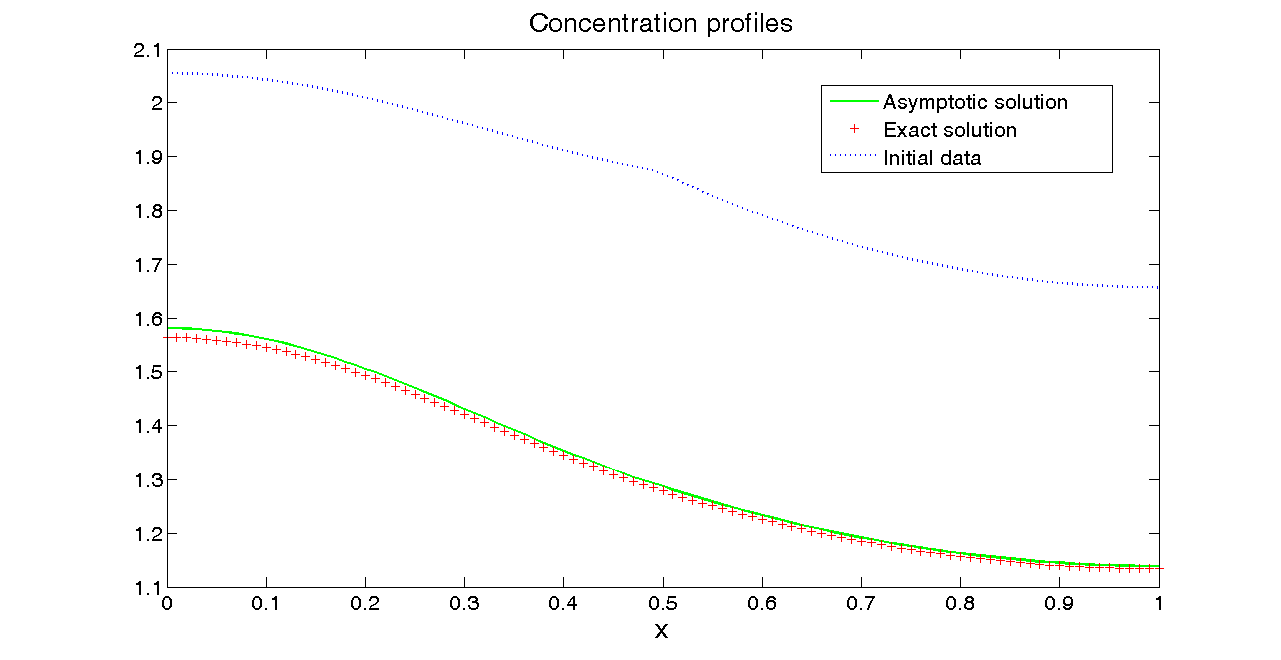}&\includegraphics[scale=0.17]{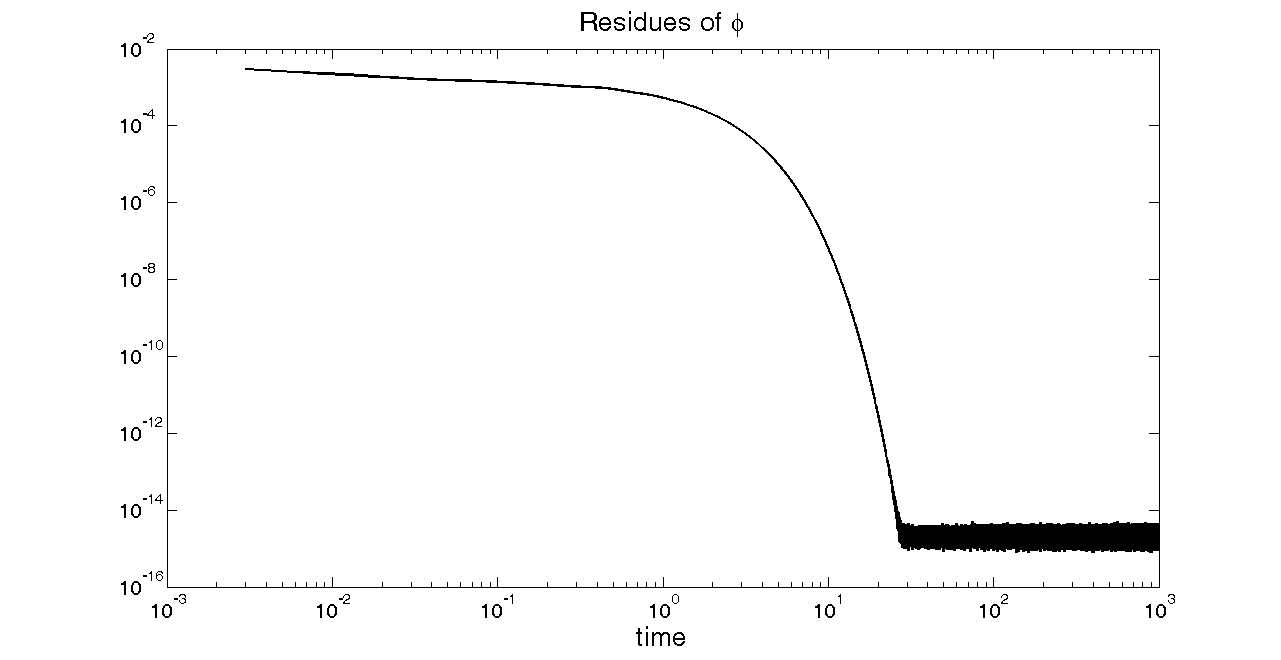}\\
\includegraphics[scale=0.17]{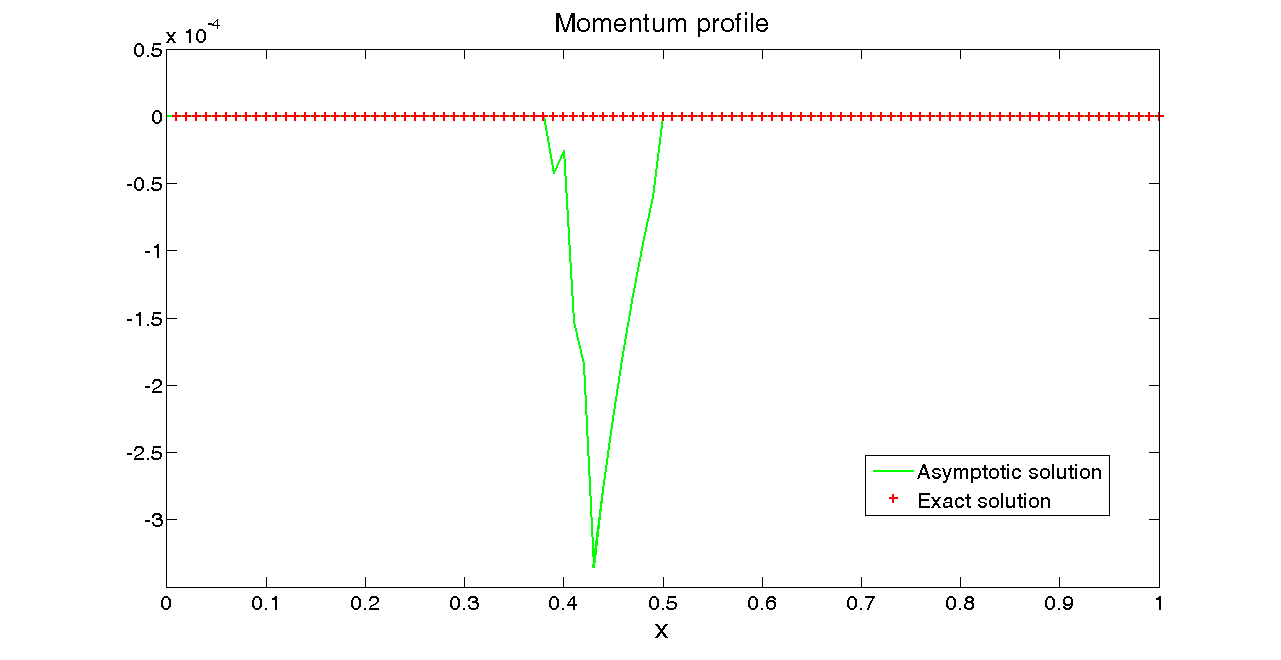}&\includegraphics[scale=0.17]{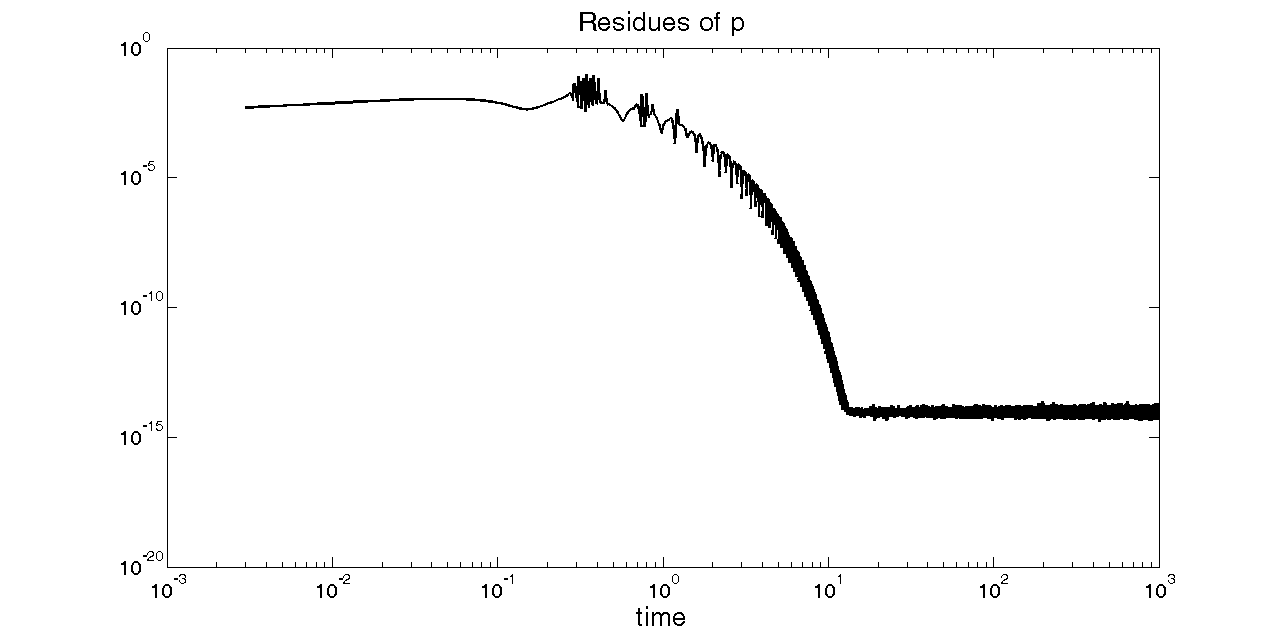}\\
\end{tabular}
\caption{{\small{On the left: Profiles for system (\ref{eq:main_system}) with $\gamma=2$ and $\varepsilon=D=a=b=L=1,\quad\chi=50$ of density (on top), concentration  (in the middle) and momentum (on bottom) : (+++ red) exact solutions (\ref{lateral_bump}); 
($\dots$ blue) initial data 
with a perturbation of the interface point $x^{*}=\bar{x}+0.1$; (--- green) asymptotic solutions corresponding to the previous initial data. On the right: corresponding residues. }}}
\label{Fig:stability_test1_plus}
\end{center}
\end{figure}
 For the first type of  perturbation with $\delta=0.1$, the results are presented at Figure~\ref{Fig:stability_test1_plus}. 
 On the left,  we can see  the profiles of density (on top), concentration (in the middle) and momentum (on bottom) for the initial perturbed data in dotted blue lines (.),
the asymptotic solution after perturbation in solid green lines (-) and the 
exact   solutions in lines marked with red crosses (+). On the right,   the corresponding  residues are displayed. Initially,  the residues of density have  values of order $10^{-2}$, which suggests that the solution is still evolving. Then, at some point, they decrease rapidly  to 
nearly $10^{-14}$ and stabilize at this value, which  means that the solution has reached asymptotically  the expected steady state. 
We remark  that the asymptotic profiles match perfectly the expected  solutions, which confirms that the stationary solution with a   
 lateral bump given by (\ref{lateral_bump})  is stable under this  kind of perturbation.

In Figure~\ref{Fig:stability_test2}, we display the results obtained with the second type of perturbation \eqref{perturb} with $[x_{1},x_{2}]=[0.6\bar{x},0.8\bar{x}]$, with the same curves as at Figure~\ref{Fig:stability_test1_plus}. 
\begin{figure}[htbp!]
 \begin{center}
 \begin{tabular}{cc}
\includegraphics[scale=0.17]{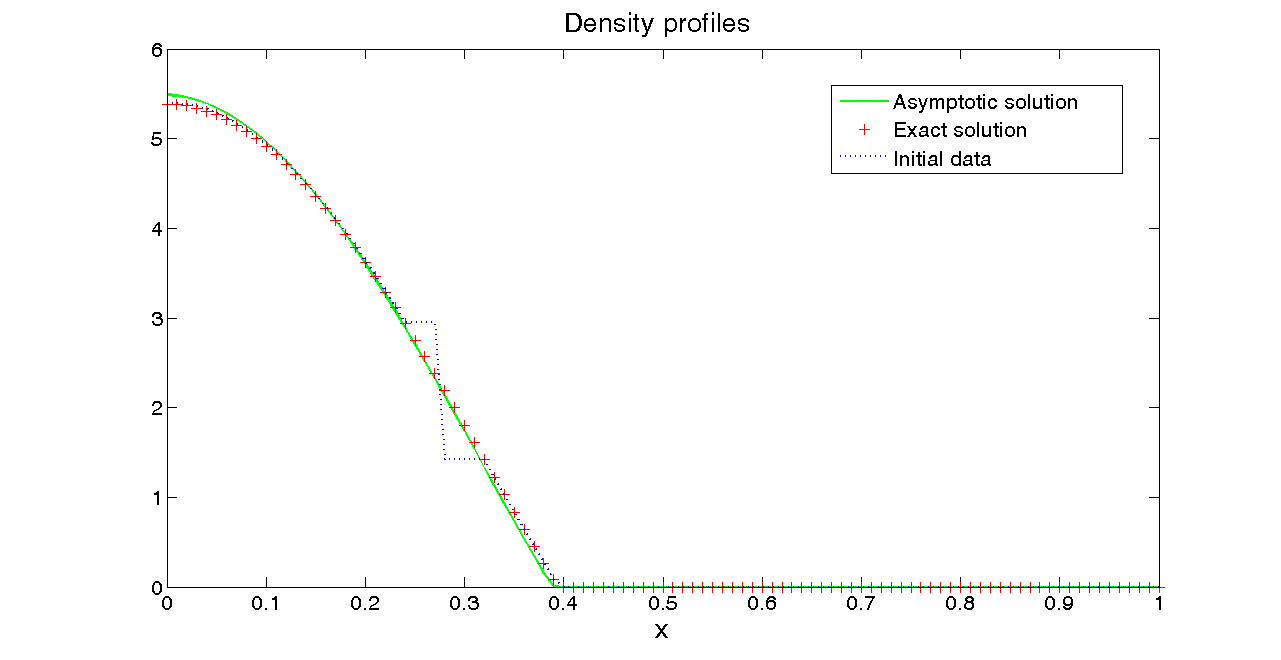}&\includegraphics[scale=0.17]{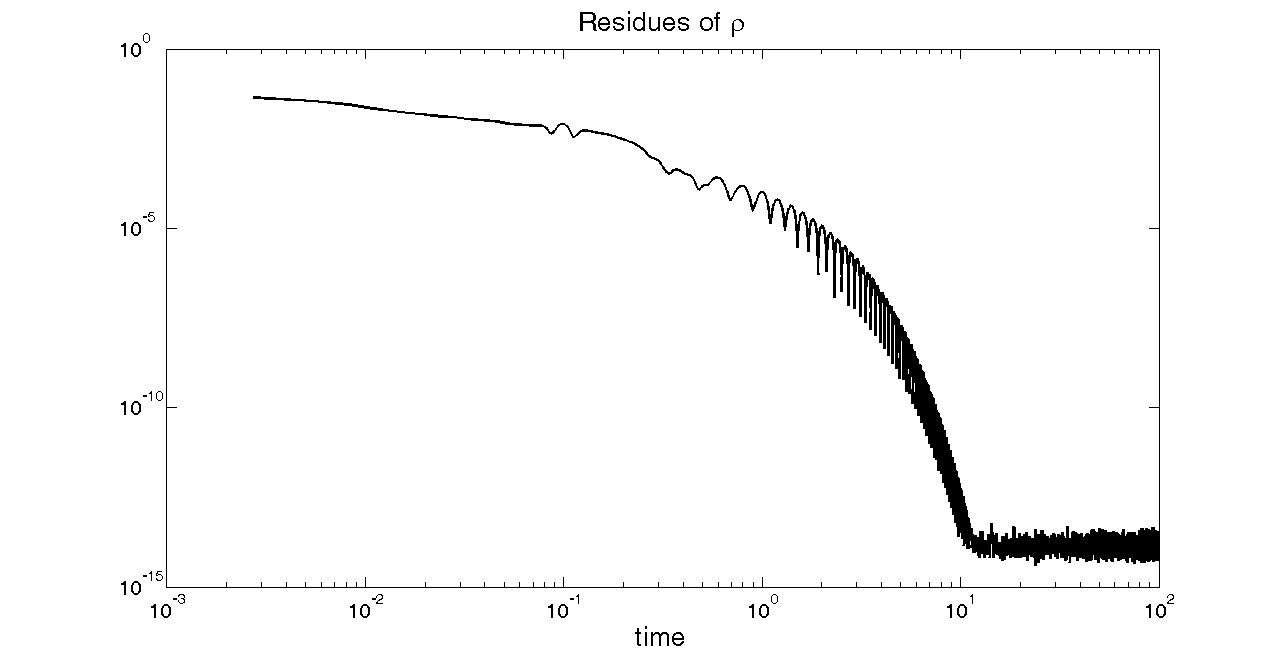}\\
\includegraphics[scale=0.17]{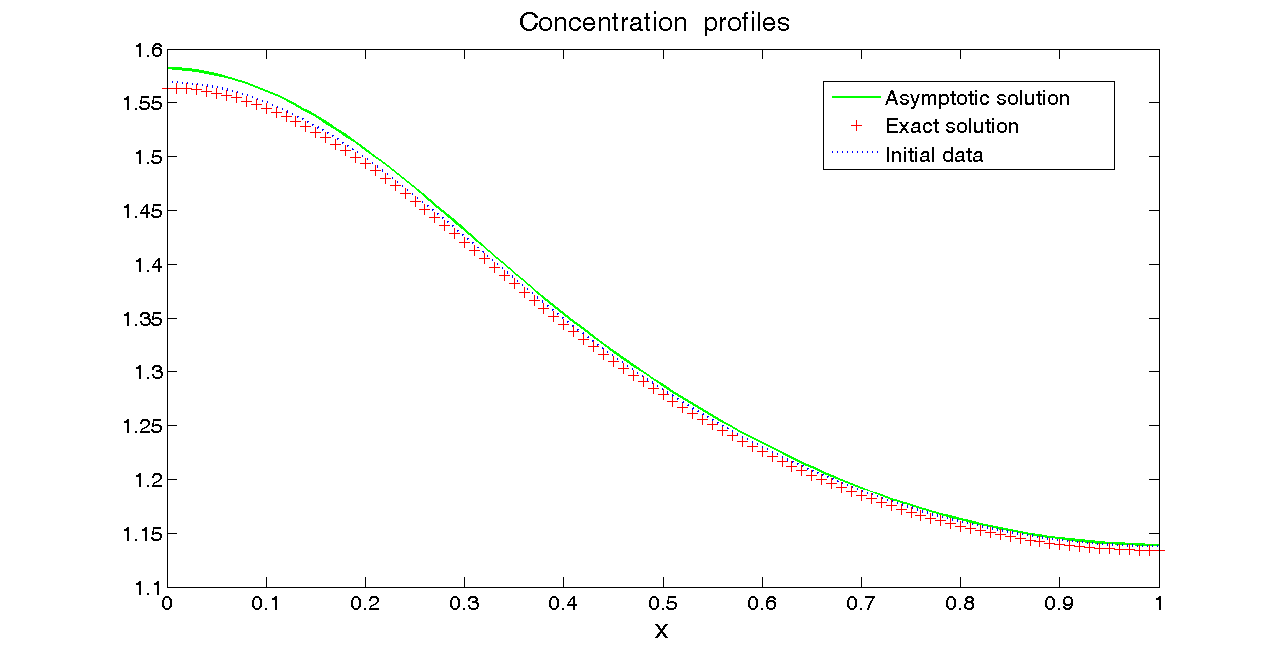}&\includegraphics[scale=0.17]{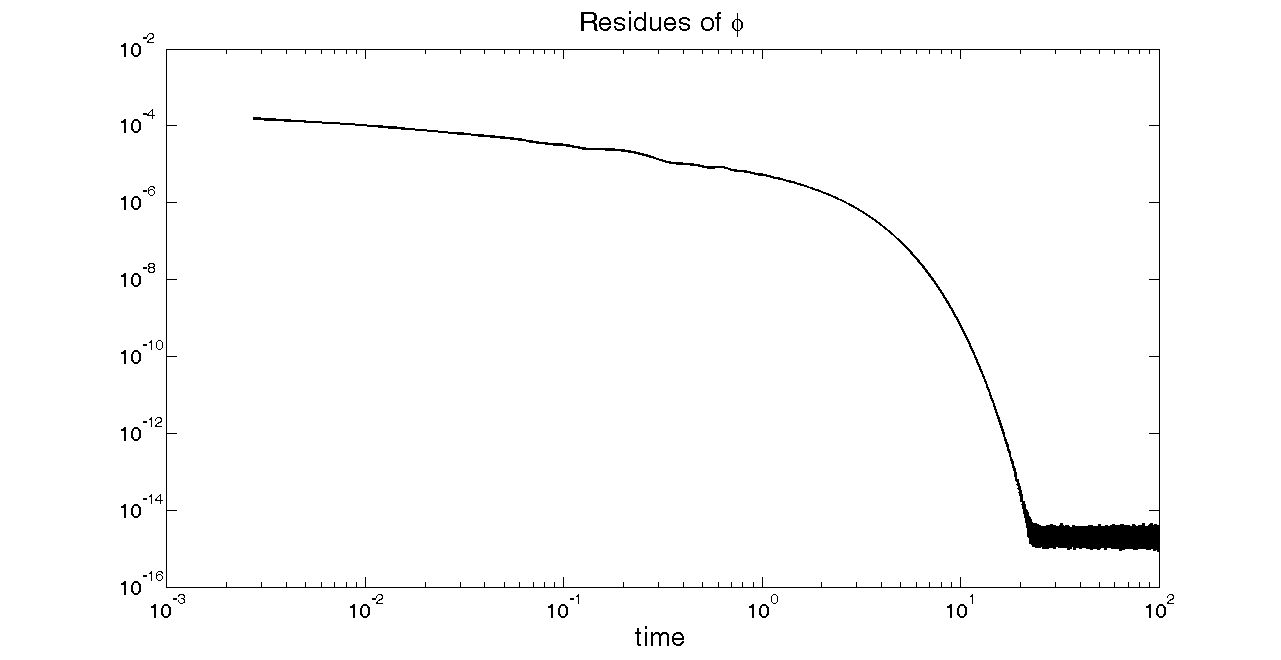}\\
\includegraphics[scale=0.17]{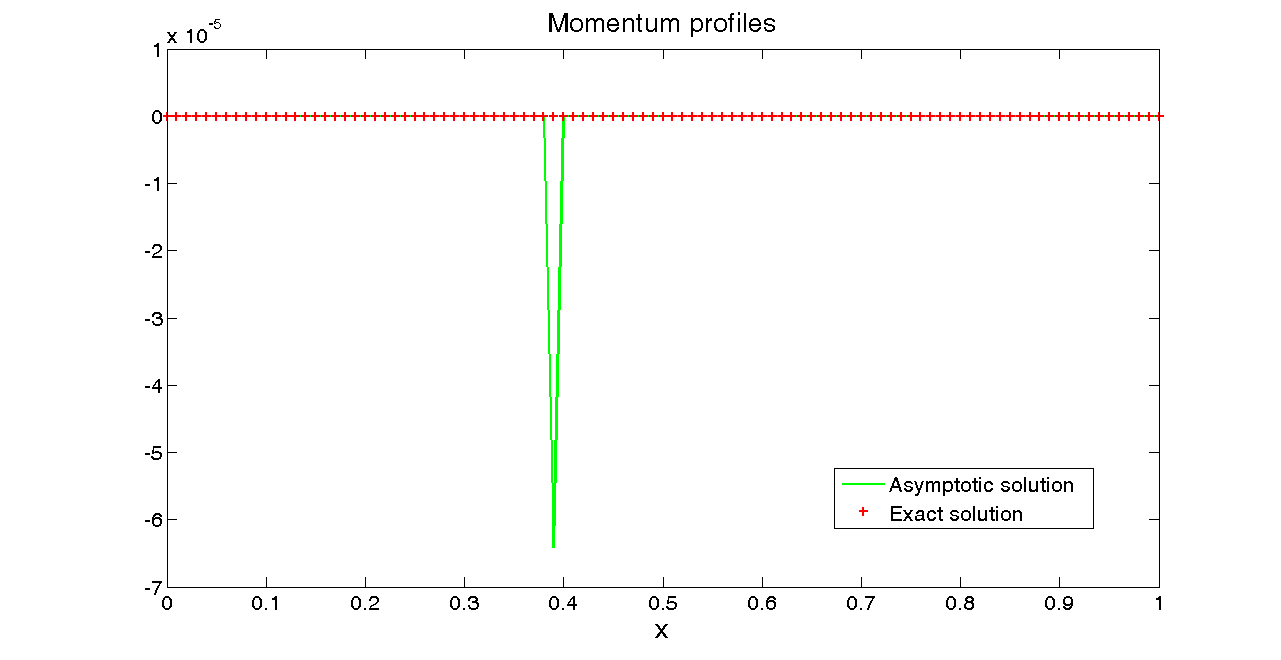}&\includegraphics[scale=0.17]{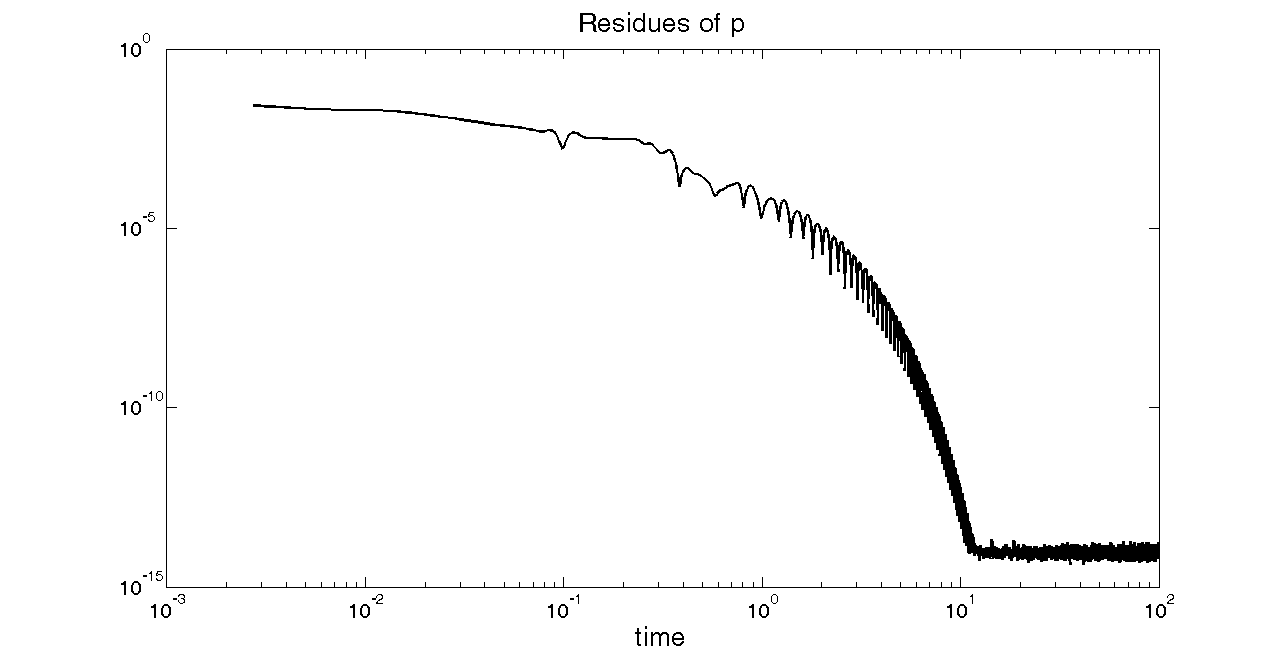}\\
\end{tabular}
\caption{{\small{On the left: Profiles for system (\ref{eq:main_system}) with $\gamma=2$ and $\varepsilon=D=a=b=L=1,\quad\chi=50$ of density (on top), concentration (in the middle)  and momentum (on bottom) : (+++ red) exact solutions (\ref{lateral_bump}); 
($\dots$ blue) initial data  perturbed in  the domain $[x_{1},x_{2}]=[0.6\bar{x},0.8\bar{x}]$;
 (--- green) asymptotic solutions corresponding to the previous initial data. On the right: corresponding residues.}}}
\label{Fig:stability_test2}
\end{center}
\end{figure}
We notice that the numerical results confirm the stability of the stationary solution composed of a lateral bump. The type of  perturbation is different, but the mechanisms of convergence to equilibrium have similar features.

\subsection{Dependence on the parameters  $\chi$ and $L$ of the asymptotic states }
Now, in this section, we study which  stationary solutions to system   (\ref{eq:main_system}) are reached asymptotically as a function of some parameters of the system and, in particular, how many bumps the asymptotic profile contains.  
 We showed analytically that, if $\tau\leq 0$, the equilibrium is given by a constant state. The same situation occurs if $\tau>0$ and $L\leq\pi\slash\sqrt{\tau}$. If we increase the size $L$ of the domain, nonconstant stationary solutions 
with several bumps may appear.  However, we are unable to determine analytically how many regions of positive density  the asymptotic solution will contain
and a numerical study is necessary here.
More precisely, our aim is to analyze how 
 the chemotactic sensitivity $\chi$ and the  length of the domain $L$  influence the number of bumps in the case of a quadratic pressure $\gamma=2$. 

In Figure~\ref{Fig:dependence_L}, we display  asymptotic profiles of the density obtained for different lengths $L$ of the domain. 
The initial data are taken equal to :
\begin{displaymath}
\rho_{0}(x)=\xi\left( 1+\sin(4\pi|x-L\slash 4|)\right),\quad\phi_{0}=0,\qquad u_{0}=0,
\end{displaymath}
where $\xi$ is computed as a function of the length $L$ of the domain in order to keep the initial mass constant and equal to $M=1.3183$.
We observe that if the length of the domain is not large enough, that is $L\leq\pi\slash\sqrt{\tau}$, the asymptotic  state is constant, as expected. Increasing the size of the domain from $L=1$ to $L=4$ 
leads to the concentration of particles at the boundary and nonconstant solutions composed of one bump appear. A further increase, from $L=7$ to $L=30$, allows new bumps to appear and a two bumps solution is observed. 
\begin{figure}[htbp!]
 \begin{center}
 \begin{tabular}{cc}
\includegraphics[scale=0.17]{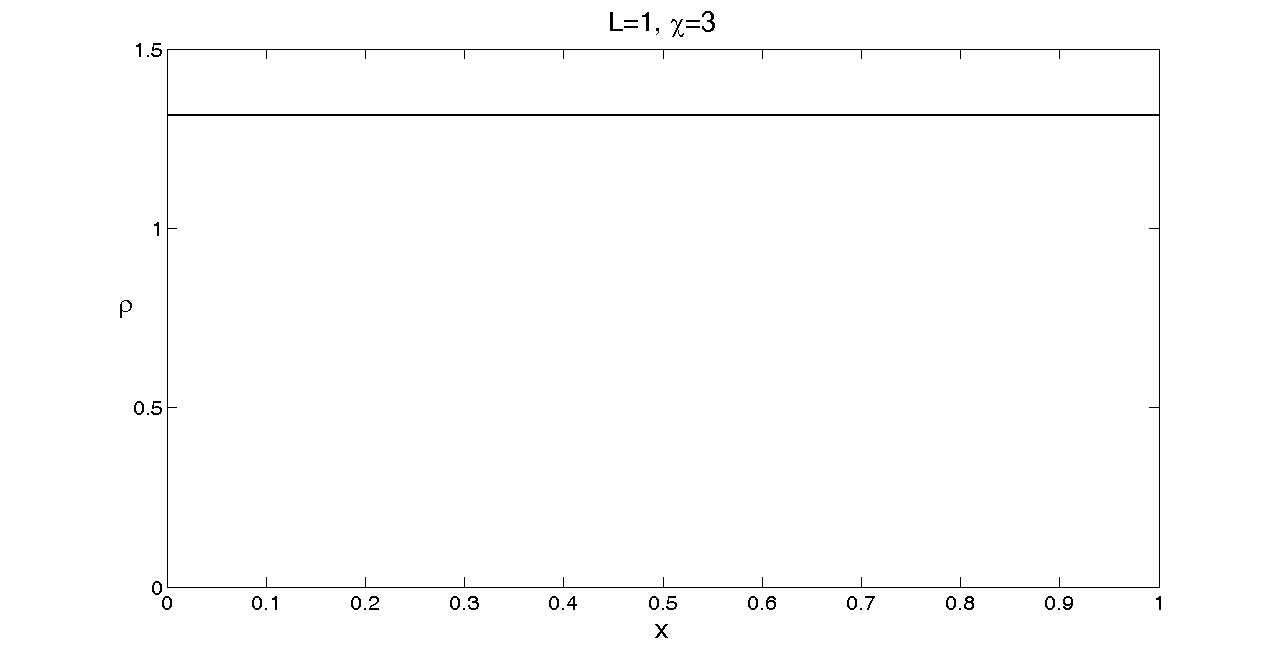}&\includegraphics[scale=0.17]{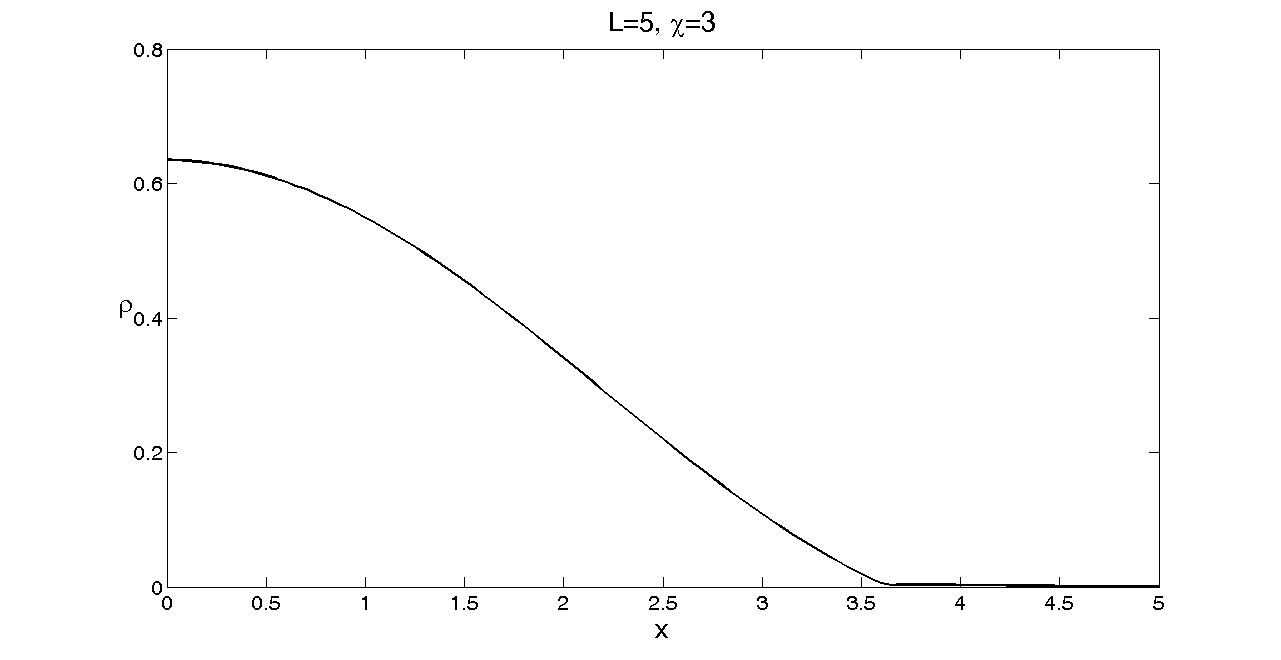}\\
$L=1$&$L=5$\\
\includegraphics[scale=0.17]{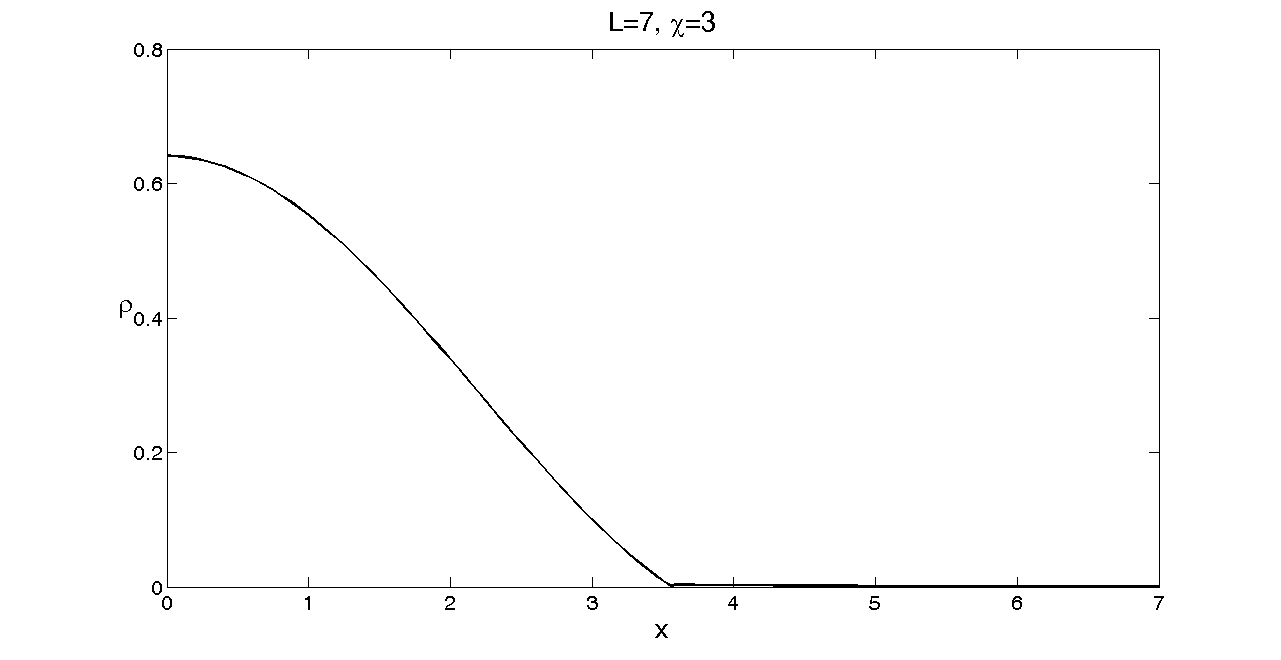}&\includegraphics[scale=0.17]{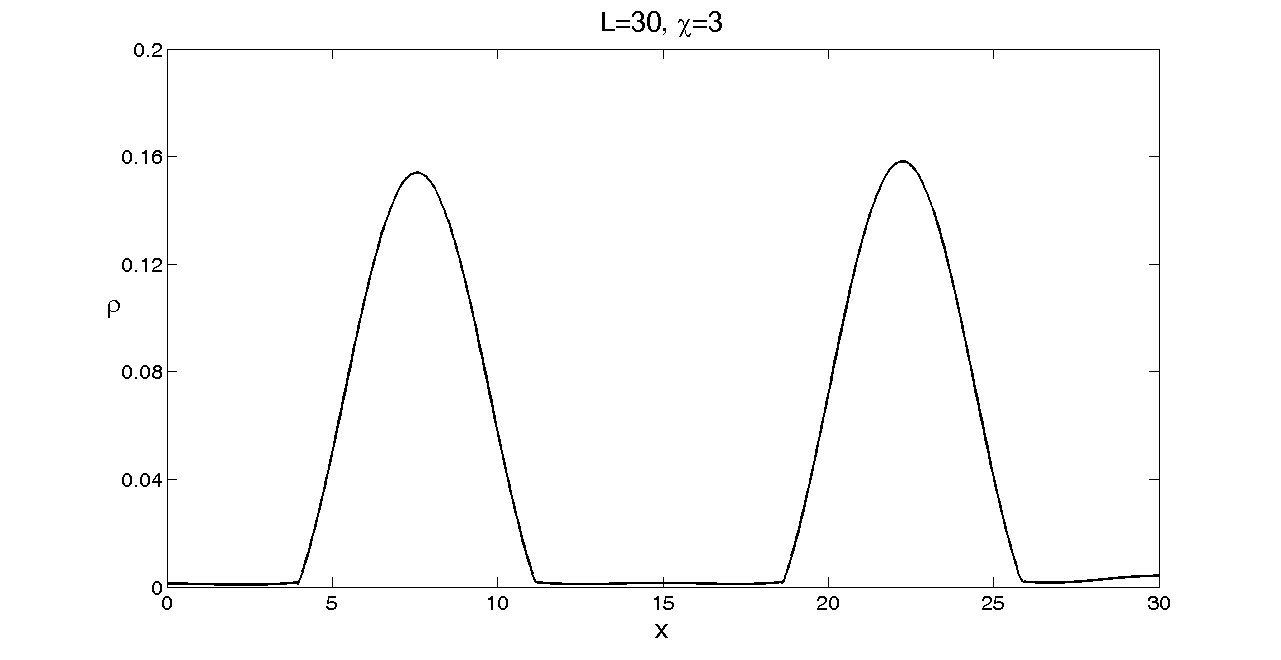}\\
$L=7$&$L=30$\\
\end{tabular}
\caption{{\small{Density profiles of  the asymptotic  states of system (\ref{eq:main_system}) with $P(\rho)=\varepsilon\rho^2$, $\chi=3$, $\varepsilon=D=a=b=1$ and total mass $M=1.3183$ 
for different values of the length of the domain $L=\{1,5,7,30\}$.}}}
\label{Fig:dependence_L}
\end{center}
\end{figure}

We observe the same phenomenon when we increase the value of the chemosensitivity constant $\chi$. We present  at Figure~\ref{Fig:dependence_chi} different asymptotic profiles obtained for different values of $\chi$. Passing from  $\chi=3$ to $\chi=5$ leads to a higher concentration of particles which, as a result, produces free space for new bumps.  However, from $\chi=5$ to $\chi=200$, the global form of the solution, i.e. two lateral bumps on each side of the domain remains the same, even if the two bumps tend to become higher and narrower. 
\begin{figure}[htbp!]
 \begin{center}
 \begin{tabular}{cc}
\includegraphics[scale=0.17]{L7_chi3.png}&\includegraphics[scale=0.17]{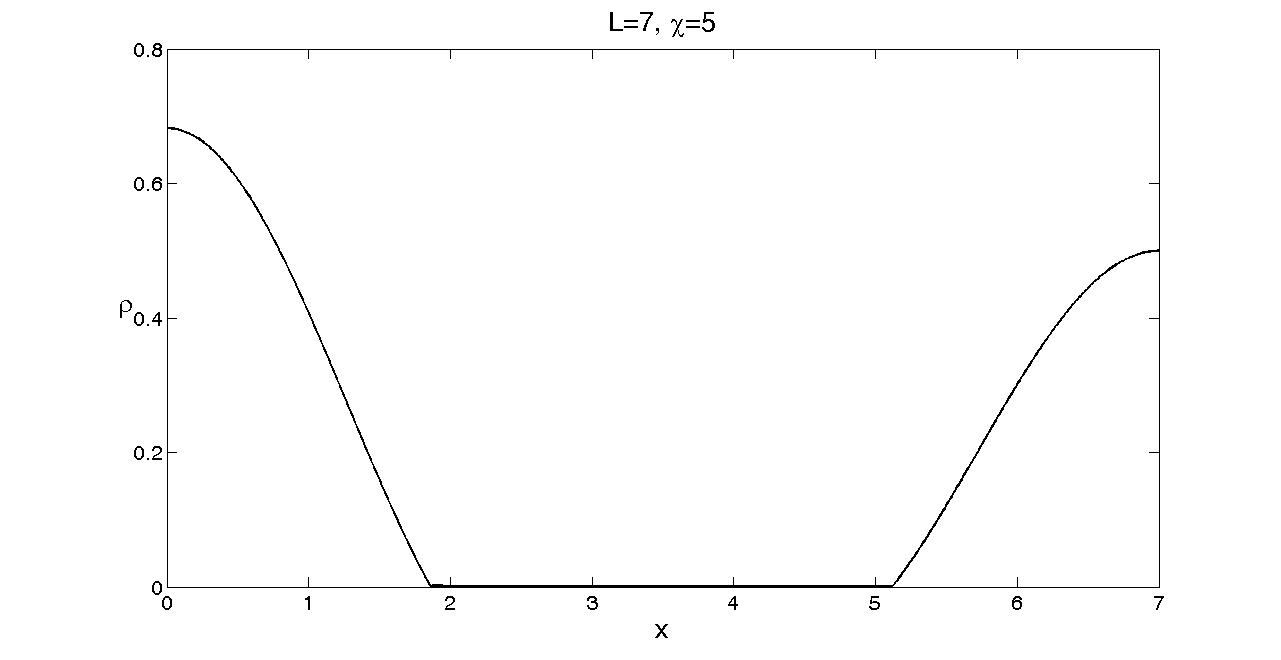}\\
$\chi=3$&$\chi=5$\\
\includegraphics[scale=0.17]{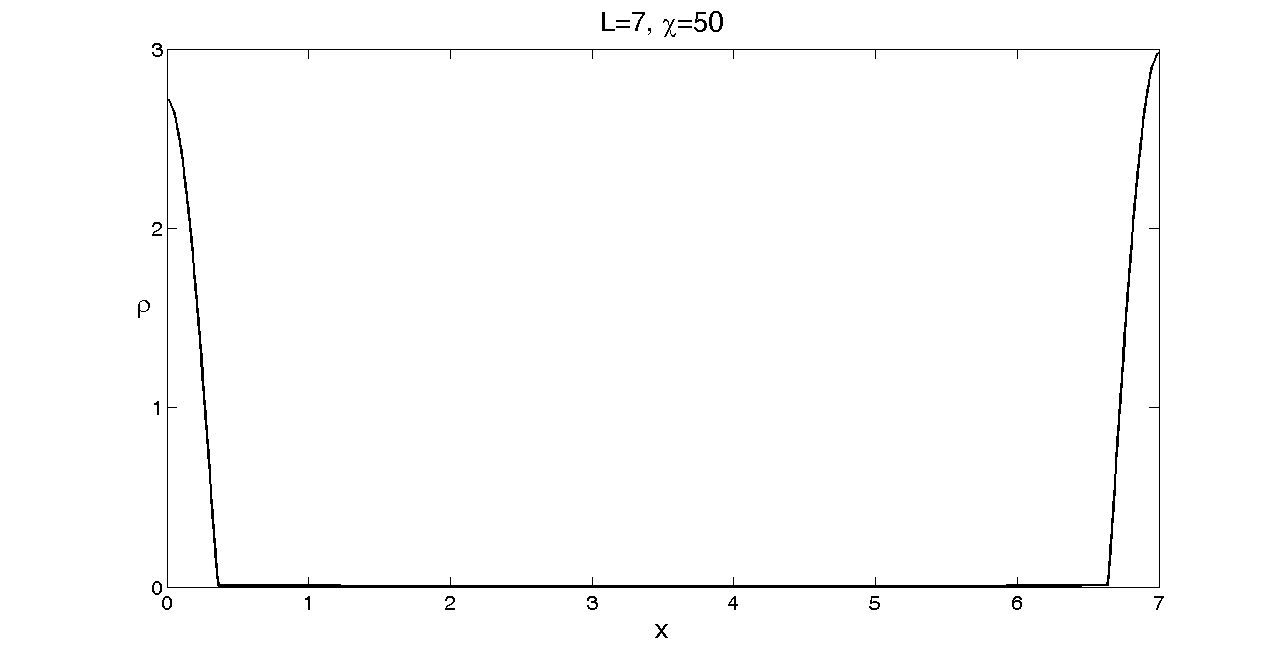}&\includegraphics[scale=0.17]{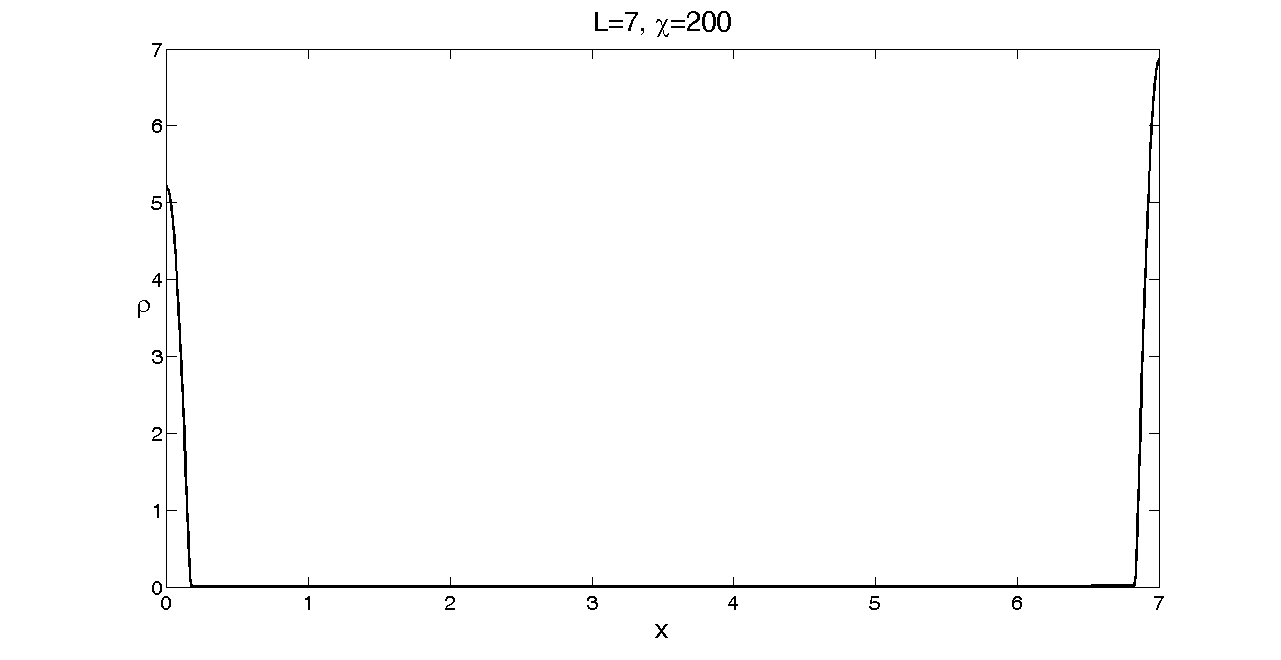}\\
$\chi=50$&$\chi=200$\\
\end{tabular}
\caption{{\small{Density profiles of  the asymptotic  states of system (\ref{eq:main_system}) with $P(\rho)=\varepsilon\rho^2$, $L=7$, $\varepsilon=D=a=b=1$ and total mass $M=1.3183$  for different values of the chemosensitivity constant $\chi=\{3,5,50,200\}$.}}}
\label{Fig:dependence_chi}
\end{center}
\end{figure}

We see that the free space available for the cells has an essential effect on the formation of nonconstant steady states. If the size of the domain  is too small, only constant solutions are expected. A growth of the available space, by increasing directly the length of the domain or by concentrating particles in smaller regions, leads to the formation of new bumps. However, at some point, a  further increase of the chemotactic force, even up to important values, 
does not produce any  new bumps.
So,  we think  that, on each domain,  there is a maximum number of bumps 
 once  all the parameters 
  are fixed.

\subsection{Dependence  on the adiabatic coefficient $\gamma$ of the asymptotic  states}
The adiabatic exponent $\gamma$ describes the response of cells to compression. 
For a high value of $\gamma$, the internal pressure repealing the cells is very strong.
For example, for ideal gases, the value of  $\gamma$ is taken  between 1 and 2, whereas in the Saint-Venant system, describing geophysical flows, it is equal to 2. In the  case of cells, it should be much larger as they are less compressible than gas or water molecules. 

In Section~\ref{sec:stationary_solutions}, we presented a general form of nonconstant steady states with several bumps 
  for arbitrary $\gamma$. However,  we obtained explicit solutions only in the case $\gamma=2$. In this subsection,  we study how the equilibria change for  different values of $\gamma$ and especially how  the number of bumps varies.
  
 
 We consider system (\ref{eq:main_system}) with $\varepsilon=1,D=0.1,a=20,b=10,\chi=10$ defined on the interval $[0,3]$
 	and with initial data
	\begin{displaymath}
\rho_{0}(x)=\left( 1.5+\sin(4\pi|x-L\slash 4|)\right),\quad\phi_{0}=0,\qquad u_{0}=0.
\end{displaymath}
  In Figure~\ref{fig:dependence_gamma}, we plot the asymptotic solutions for density (on the left) and chemoattractant concentration (on the right) 
   for the following values of the adiabatic coefficient : $\gamma=$\{2 (blue), 3 (green), 4 (red), 5 (cyan)\}. 
    We observe how the number of bumps changes, namely $4$ bumps for $\gamma=2$, $3$ for  $\gamma=3$, $2$ for  $\gamma=4$ (among which $2$ lateral bumps in these $3$ cases) and a constant profile for $\gamma=5$. Large values of $\gamma$ imply strong repealing forces at higher densities. It prevents the formation  of  high concentrations of cells. When the value of $\gamma$ increases, the 
    height 
    of bumps decreases, while their support enlarges. Moreover, when the distance between the supports of two neighboring bumps becomes zero, they join together and,  for $\gamma$ high enough, the pressure forces are stronger than the chemotactic movement and constant steady states are observed. 
\begin{figure}[htbp!]
\begin{center}
\begin{tabular}{cc}
\includegraphics[scale=0.17]{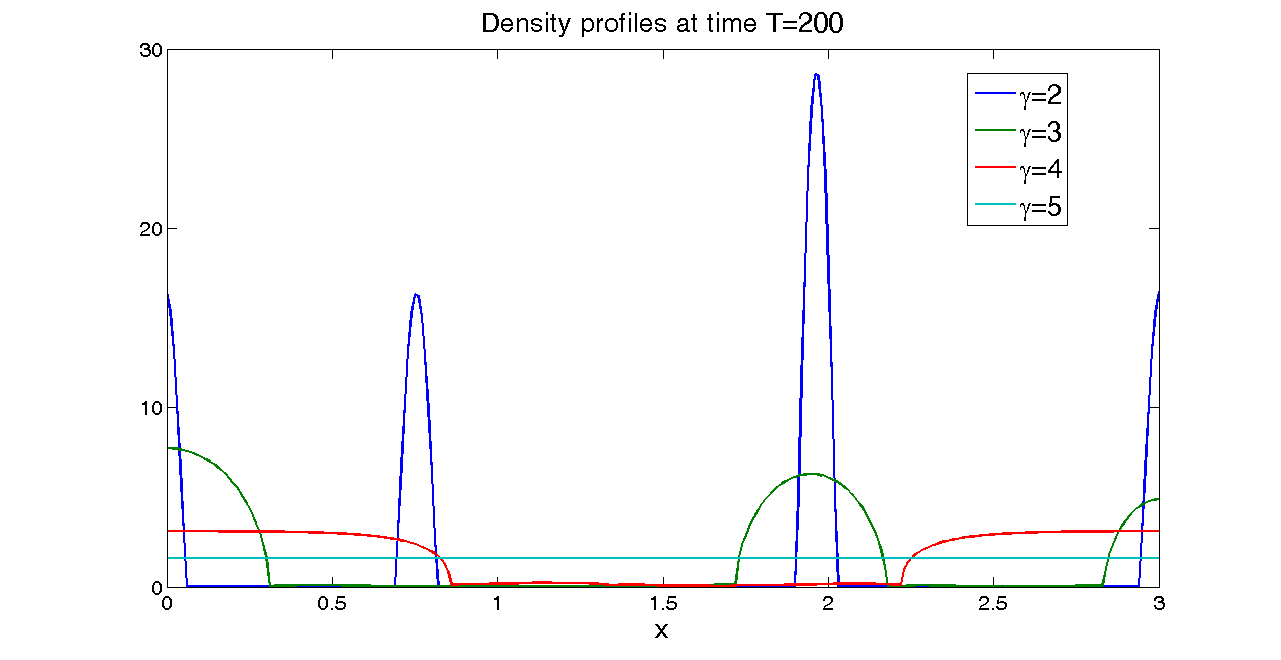}&
\includegraphics[scale=0.17]{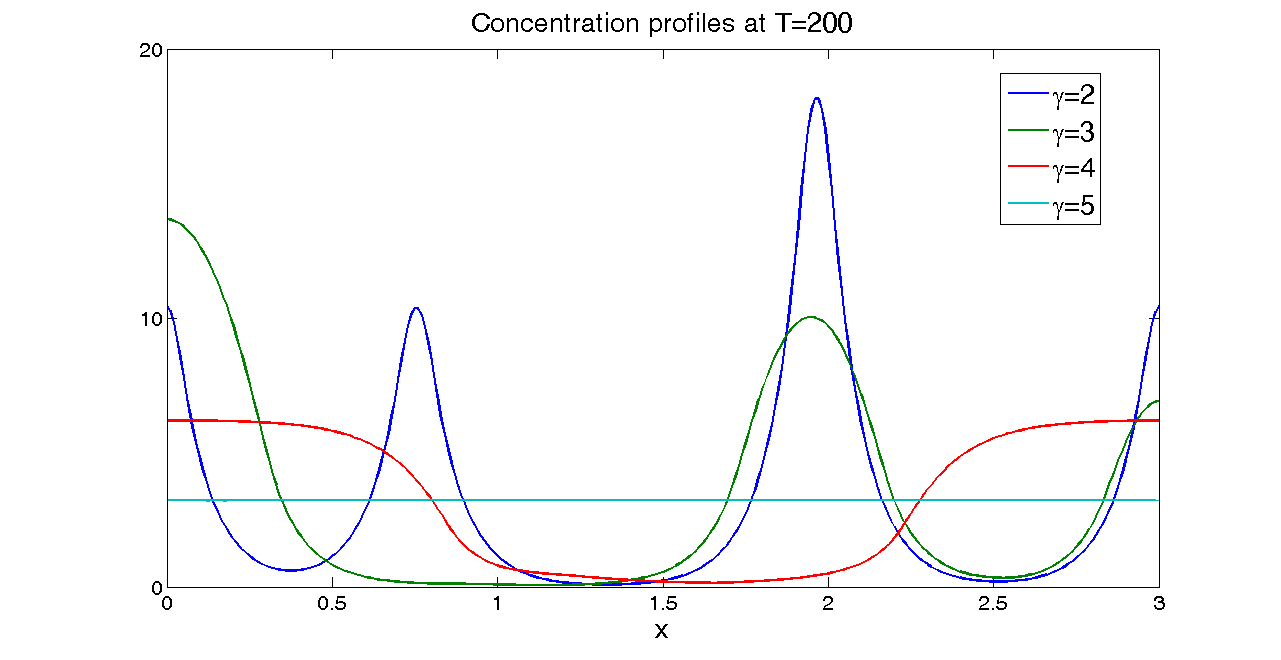}
\end{tabular}
\caption{{\small{Profiles of the density $\rho$ (on the left) and of the chemoattractant  concentration $\phi$  (on the right) at asymptotic  states of  system (\ref{eq:main_system}) with $\varepsilon=1,D=0.1,a=20,b=10,\chi=10$ on the interval $[0,3]$ 
for different values of the adiabatic exponent $\gamma=$\{2 (blue),3 (green),4 (red),5 (cyan)\}.}}}
\label{fig:dependence_gamma}
\end{center}
\end{figure}

\subsection{Dependence  on the initial mass of the asymptotic  states}
Experiments with endothelial cells performed by Serini et.al \cite{Gamba2} showed that a vascular-like network develops only if the initial density of cells ranges from $100$ to $400$ cells/mm$^2$. Below this interval, a disconnected structure is observed, while, above, a continuous carpet of cells with holes appears. System (\ref{eq:main_system}) was constructed to model the phenomenon of early formation of blood vessels and its solutions should also reflect this  dependence on the initial mass. These experimental results suggest a particular behavior of steady states  with several bumps 
of system (\ref{eq:main_system}). The regions where the density is strictly positive may  correspond to the location of capillaries and  would become thicker for large initial masses. Moreover, we expect to find a threshold value of the initial mass, above which constant equilibria are observed. 

We performed  some simulations of system (\ref{eq:main_system}) with $\gamma=2$ and $\gamma=3$, which are presented at Figure~\ref{mass}. On the left (resp. on the right), we can see the asymptotic profiles for density  in the case $\gamma=2$ (resp. $\gamma=3$)  for different values of the initial mass. The initial data are the same equal to
\begin{displaymath}
\rho_{0}(x)=\xi\left( 1+\sin(4\pi|x-L\slash 4|)\right),\quad\phi_{0}=0,\qquad u_{0}=0,
\end{displaymath}
and the initial density  is only multiplied  by different constants in order to change the value of the initial mass. 
  In the first case $\gamma=2$, we notice that the number of bumps is equal to $4$ and remains exactly the same for the different masses. Moreover, 
  the support of the  bumps is also independent of the initial mass, whereas the height of the bumps increases with the initial mass. This can be seen theoretically, since the equations (\ref{bump:x}) or  (\ref{interface_centered}) determining the interface point $\bar{x}$ do not depend on the mass, while the other equations  to determine the stationary density do. We notice that this behavior is not the one expected if we consider the experimental observations.
  However, in the second case  $\gamma=3$, the dependency on the mass is totally different and fits the experiments mentioned above.
  Indeed, when the mass increases, the supports of the bumps become larger and the bumps join together,  until reaching the constant equilibrium for a mass large enough.
\begin{figure}[htbp!]
\begin{center}
\begin{tabular}{cc}
\includegraphics[scale=0.17]{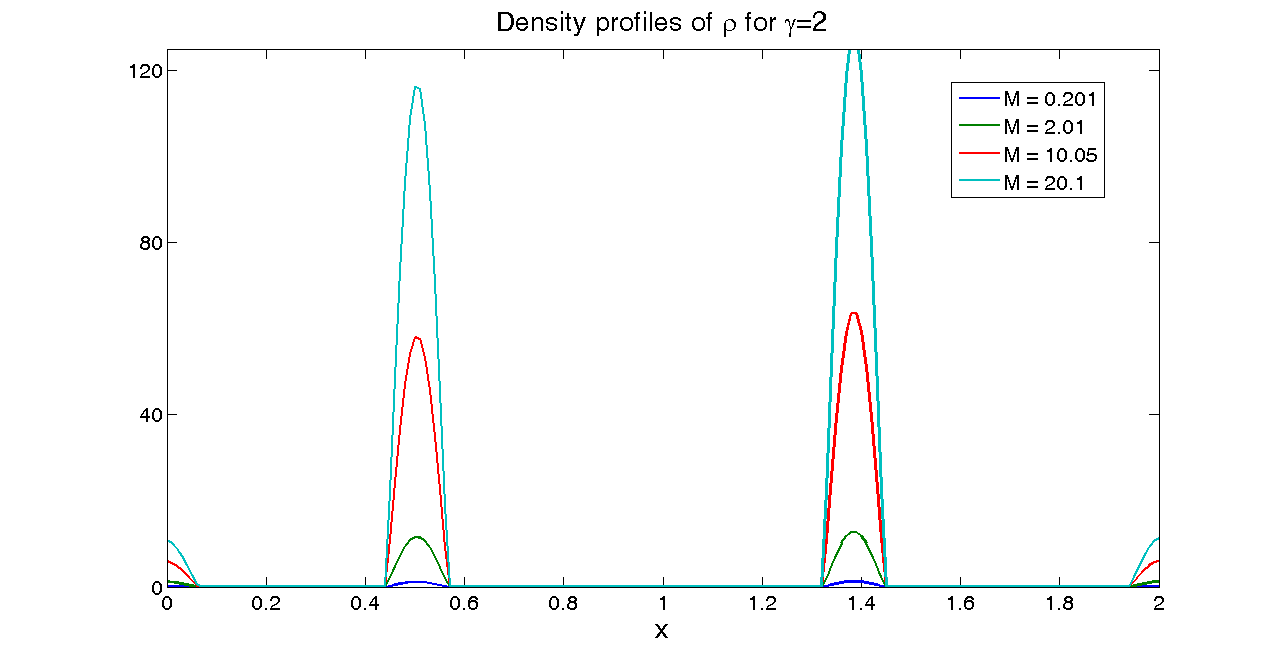}&\includegraphics[scale=0.17]{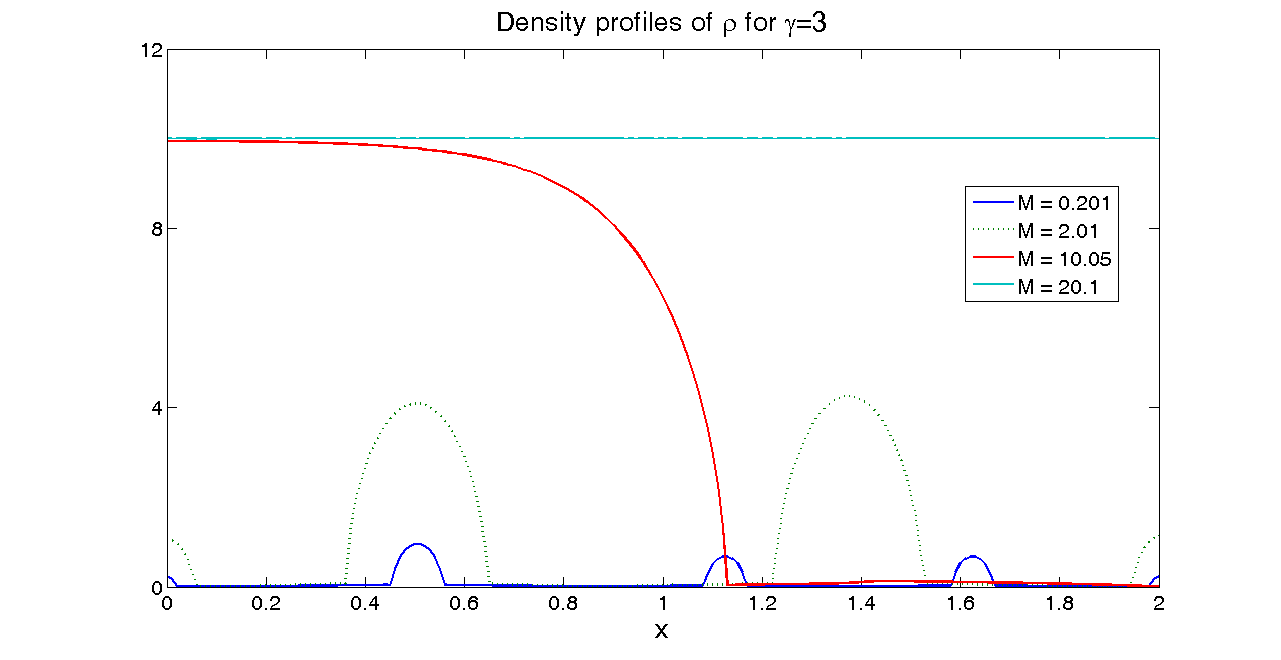}
\end{tabular}
\caption{{\small{Density profiles in the case $P(\rho)=\rho^2$ (on the left) and $P(\rho)=\rho^3$ (on the right) for system (\ref{eq:main_system}) with $\varepsilon=1$, $D=0.1$, $a=20$, $b=10$, $\chi=10$  and an initial datum of the form  $\rho_{0}(x)=\xi(1+\sin(4\pi|x-L\slash 4|))$. Comparison for different initial masses with  $\xi=$\{0.1(blue), 1 (green), 5 (red), 10 (cyan)\}.}}}
\label{mass}
\end{center}
\end{figure} 

\medskip

{\bf Acknowledgement.} The authors thank Fran\c{c}ois Bouchut for some useful suggestions. This work has been partially supported by
the project PORAbruzzo and   by the ANR project MONUMENTALG, ANR-10-JCJC 0103.


\end{document}